\documentclass[twoside]{article}

\usepackage{aistats2019}

\usepackage{booktabs}
\usepackage{multirow} 
\usepackage[table,dvipsnames]{xcolor}
\usepackage{colortbl}
\definecolor{LightCyan}{rgb}{0.88,1,1}
\newcommand*{\belowrulesepcolor}[1]{%
	\noalign{%
		\kern-\belowrulesep
		\begingroup
		\color{#1}%
		\hrule height\belowrulesep
		\endgroup
	}%
}
\newcommand*{\aboverulesepcolor}[1]{%
	\noalign{%
		\begingroup
		\color{#1}%
		\hrule height\aboverulesep
		\endgroup
		\kern-\aboverulesep
	}%
}

\newcommand{\CubicBound}{C}
\usepackage{natbib}

\usepackage{microtype}
\usepackage{amsmath,amssymb,amsfonts,amstext,bm}
\usepackage{latexsym,graphicx,epsf,epsfig,psfrag}
\usepackage{subfigure}
\usepackage{dsfont}
\usepackage{booktabs} % for professional tables

\usepackage{enumitem} 
\usepackage{caption} 
\usepackage{multicol} 
\usepackage[title]{appendix}

\allowdisplaybreaks

\usepackage{algorithm}
\usepackage{algorithmic}

\usepackage{mathtools}

\DeclarePairedDelimiter{\floor}{\lfloor}{\rfloor}

\usepackage{amsthm}
\usepackage[colorlinks = true]{hyperref}
\usepackage{cleveref}
\hypersetup{linkcolor =red, citecolor = blue}

\newcommand{\BlackBox}{\rule{1.5ex}{1.5ex}}  % end of proof

\newtheorem{thm}{Theorem}
\newtheorem{lemma}[thm]{Lemma} 
\newtheorem{proposition}[thm]{Proposition}

\newtheorem{assum}{Assumption}
\newcommand{\Ebb}{\mathbb{E}}

\newcommand{\xb}{\mathbf{x}}

\newcommand{\A}{\mathbf{A}}

\newcommand{\Hb}{\mathbf{H}}
\newcommand{\I}{\mathbf{I}}
\newcommand{\U}{\mathbf{U}}
\newcommand{\X}{\mathbf{X}}
\newcommand{\Y}{\mathbf{Y}}
\newcommand{\Z}{\mathbf{Z}}

\newcommand{\g}{\mathbf{g}}
\newcommand{\s}{\mathbf{s}}
\newcommand{\x}{\mathbf{x}}
\newcommand{\ub}{\mathbf{u}}

\newcommand{\bLambda }{\mathbf{\Lambda}}

\newcommand{\RR}{\mathbb{R}}

\newcommand{\Oc}{\mathcal{O}}

\newcommand{\zero}{\mathbf{0}}

\newcommand{\argmin}{\mathop{\mathrm{argmin}}}

\newcommand{\norm}[1]{\| #1 \|}
\newcommand{\norml}[1]{\left\| #1 \right\|}
\newcommand{\normlarge}[1]{\left\Vert #1\right\Vert}
\newcommand{\tre}{\mathrm{ tr} \ \mathrm{ exp}}

\newcommand{\numleqslant}[1]{\overset{\text{(#1)}}{\leqslant}}
\newcommand{\numequ}[1]{\overset{\text{(#1)}}{=}}

\newcommand\numberthis{\addtocounter{equation}{1}\tag{\theequation}}

\usepackage{tabu}

\begin{document}
\twocolumn[

\aistatstitle{Stochastic Variance-Reduced Cubic Regularization for Nonconvex Optimization} 

\aistatsauthor{ Zhe Wang  \And Yi Zhou  \And  Yingbin Liang \And Guanghui Lan }

\aistatsaddress{ Ohio State University \\ wang.10982@osu.edu \And Ohio State University \\zhou.1172@osu.edu \And Ohio State University \\ liang.889@osu.edu \And Georgia Institute of Technology \\george.lan@isye.gatech.edu}
]

\begin{abstract} 
Cubic regularization (CR) is an optimization method with emerging popularity due to its capability to escape saddle points and converge to second-order stationary solutions for nonconvex optimization. However, CR encounters a high sample complexity issue for finite-sum problems with a large data size. %Various inexact variants of  CR have been proposed to improve the sample complexity.
In this paper, we propose a stochastic variance-reduced cubic-regularization (SVRC) method under random sampling, and study its convergence guarantee as well as sample complexity. We show that the iteration complexity of SVRC for achieving a second-order stationary solution within $\epsilon$ accuracy is $O(\epsilon^{-3/2})$, which matches the state-of-art result on CR types of methods. Moreover, our proposed variance reduction scheme significantly reduces the per-iteration sample complexity. The resulting total Hessian sample complexity of our SVRC is ${\Oc}(N^{2/3} \epsilon^{-3/2})$, which outperforms the state-of-art result by a factor of $O(N^{2/15})$. We also study our SVRC under random sampling without replacement scheme, which yields a lower per-iteration sample complexity, and hence justifies its practical applicability. 
\end{abstract}
 
\vspace{-4mm}
\section{Introduction}
\vspace{-.1cm}
Many machine learning problems are formulated as finite-sum nonconvex optimization problems that take the form 
\begin{align}
\min_{\xb \in \RR^d} F(\xb) \triangleq \frac{1}{N} \sum_{i=1}^{N} f_i(\xb),  \label{eq: obj}
\end{align} 
where each component function $f_i$ corresponds to the loss on the $i$-th data sample.
While finding global optimal solutions of generic nonconvex optimization problems are challenging, various nonconvex problems in the form of \cref{eq: obj} have been shown to possess good landscape properties that facilitate convergence. For example, the square loss of a shallow linear neural network is shown to have only strict saddle points other than local minimum \citep{Baldi_1989_1}. The same property also holds for some other nonconvex problems such as phase retrieval \citep{Sun_2017} and matrix factorization \citep{RongGe_2016,Bhojanapalli_2016}. Such a remarkable property has motivated a growing research interest in designing algorithms that can escape strict saddle points and have guaranteed convergence to local minimum, and even to global minimum for problems without spurious local minimum. 

Various algorithms have been designed to have the capability to escape strict saddle points in nonconvex optimization. Such a desired property requires that the obtained solution $\xb^\star$ satisfies the second-order stationary conditions within an $\epsilon$ accuracy, i.e.,
 \begin{align} \label{eq: 2nd_opt}
 	 \norm{\nabla F(\xb^\star)} \leqslant  \epsilon, \quad \quad \nabla^2F(\xb^\star) \succcurlyeq -\sqrt{\epsilon} \bm{I}.
 \end{align}   
Therefore, upon convergence, the gradient is guaranteed to be close to zero and the Hessian is guaranteed to be almost positive semidefinite, which thresh-out the possibility to converge to strict saddle points. Among these algorithms (which are reviewed in related work), the cubic-regularized Newton's method (also called cubic regularization or CR) \citep{Nesterov2006} is a popular method that provides the second-order stationary guarantee for the obtained solution. At each iteration $k$, CR solves a sub-problem that approximates the objective function in \cref{eq: obj} with a cubic-regularized second-order Taylor's expansion at the current iterate $\xb_k$. In specific, % denoting $\s_{k+1} := \x_{k+1} - \x_k$,
 the update rule of CR can be written as
 \begin{align}
 	&\s_{k+1} = \argmin_{\s \in \RR^d} \nabla F(\x_k)^\top\s + \frac{1}{2} \s^\top \nabla^2 F(\x_k)\s +\frac{M}{6}\norml{\s}^3 , \nonumber\\
 	&\x_{k+1} = \x_k + \s_{k+1}. \label{cubic_regularization}
 \end{align}
It has been shown that CR converges to a point satisfying the second-order stationary condition  (\cref{eq: 2nd_opt}) within $\Oc( \epsilon^{-3/2})$ number of iterations. However, fully solving the exact cubic sub-problem in \cref{cubic_regularization} requires a high computation complexity, especially due to the computation of the Hessian matrices for loss functions on all the data samples.  To evaluate the complexity of CR type algorithms, we define the stochastic Hessian oracle (SHO) as follows. Given a point $\x$ and the component number $i$, the oracle returns the corresponding Hessian $\nabla^2 f_i(\x)$. Moreover, we define the subproblem oracle (SO) as a subroutine, which for a given a point $\x$,  returns the minimizer of \cref{cubic_regularization}. In \cite{Cartis2011a}, the authors proposed an inexact cubic-regularized (inexact-CR) Newton's method, which formulates the cubic sub-problem in \cref{cubic_regularization} with an inexact Hessian $\Hb_k$ that satisfies
\begin{align}
	\norml{(  \Hb_k - \nabla^2 F(\x_k))   \s_{k+1}} \leqslant C \norml{  \s_{k+1}}^2, \label{eq: inexact_H}
\end{align}
where $C \geqslant 0$ is a certain numerical constant. In particular, \cite{Cartis2011a} showed that such an inexact method achieves the same order of theoretical guarantee as the original CR. This inexact condition has been explored  in various   situations \citep{kohler2017,Cartis2012b,Cartis2012,Yi2018}. Especially, in order to satisfy the inexact Hessian condition in \cref{eq: inexact_H}, \cite{kohler2017} proposed a practical sub-sampling scheme (referred to SCR) to implement  the inexact-CR. Specifically, at each iteration $k$,  SCR collects two index sets $\xi_g(k), \xi_H(k)$ whose elements are sampled uniformly from $\{1, \ldots, N\}$ at random, and then evaluates respectively the gradients and Hessians of the corresponding component functions, i.e., $\g_k \triangleq 	\frac{1}{|\xi_g(k)|}\sum_{i \in \xi_g(k)} \nabla f_i(\x_k)$ and $\Hb_k \triangleq 	\frac{1}{|\xi_H(k)|}\sum_{i \in \xi_H(k)} \nabla^2 f_i(\x_k)$. Then,  SCR solves the following cubic sub-problem at the $k$-th iteration.
 \begin{align*}
 	\s_{k+1} &= \argmin_{\s \in \mathbb{R}^d} \g_k^\top \s + \tfrac{1}{2} \s^\top  \Hb_k \s +\tfrac{M}{6}\norml{\s}^3.
 \end{align*}
%Under a high probability guarantee, 
\cite{kohler2017} showed that if the mini-batch sizes to satisfy
\begin{align} \label{subsampling_sample_size}
	|\xi_{g}(k)| \geqslant \mathcal{O}\left(\frac{1}{ \norml{\s_{k+1}}^{4}}\right),    
	|\xi_{H}(k)| \geqslant \mathcal{O}\left(\frac{1}{ \norml{\s_{k+1}}^{2}}\right),
\end{align}
then the sub-sampled mini-batch of Hessians $\Hb_k$ satisfies \cref{eq: inexact_H} and the sub-sampled mini-batch of gradients $\g_k$ satisfies 
\begin{align}
	\norml{   \g_k - \nabla F(\x_k) } \leqslant C_1 \norml{  \s_{k+1}}^2, \label{eq: inexact_g} 
\end{align}
where $C_1 \geqslant 0$ is a certain numerical constant, which further guarantee the same convergence rate for SCR as that the original exact CR.
%They showed that these inexact criterions further lead to the same convergence guarantee as that of the original exact CR method. In this way,  they provided a   computationally more efficient algorithm to solve the problem.
%
%However, in order to satisfy the inexact conditions \cref{eq: inexact_H,eq: inexact_g}, SCR \citep{kohler2017} requires the mini-batch sizes to satisfy
%\begin{align} \label{subsampling_sample_size}
%	|\xi_{g}(k)| \geqslant \mathcal{O}\left(\frac{1}{ \norml{\s_{k+1}}^{4}}\right),    
%	|\xi_{H}(k)| \geqslant \mathcal{O}\left(\frac{1}{ \norml{\s_{k+1}}^{2}}\right). 
%\end{align}

Three important issues here motivate our design of a new sub-sampling CR algorithm. 
\begin{list}{$\bullet$}{\topsep=0.ex \leftmargin=0.13in \rightmargin=0.in \itemsep =0.in}

\item It can be seen from \cref{subsampling_sample_size} that as the algorithm converges, i.e., $\s_{k+1} \to \zero$, the required sample size of SCR in  \cite{kohler2017} grows polynomially fast, resulting significant increase in computational complexity. Thus, an important open issue here is to design an improved sub-sampling CR algorithm that reduces the sample complexity (and correspondingly computational complexity) particularly when the algorithm approaches to convergence.

\item Another reason for the above pessimistic bound is because that \cite{kohler2017} analyzed the sample complexity for sampling {\em with} replacement, whereas in practice sampling {\em without} replacement can potentially have much lower sample complexity. As a clear evidence, the sample complexity for sampling {\em with} replacement to achieve a certain accuracy can be {\em unbounded}, whereas this for sampling {\em without} replacement can only be as large as the total sample size. Thus, the second open issue is to develop bounds for sampling {\em without} replacement in order to provide more precise guidance for sub-sampled CR methods.

\item We also observe that \cref{eq: inexact_H,eq: inexact_g} involve $\norml{\s_{k+1}}$ (and hence $\x_{k+1}$), which is not available at iteration $k$. \cite{kohler2017} used $s_k$ to replace $s_{k+1}$ in experiments but not theory. A more recent study \cite{Wang2018inexact} theoretically justified such a replacement with the convergence analysis, but not for stochastic sub-sampling scheme, for which the convergence analysis requires considerable efforts.

%established the convergence of CR under such an implementable inexact condition  
%\begin{align}
%\norml{   \Hb_k - \nabla^2 F(\x_k)  } \leqslant C \norml{  \s_{k}} . \label{new_inexact_condition}
%\end{align}
% We adopt such replacement  in the design of the proposed algorithm (SVRC) with mild transformation and   establish  the convergence guarantee of SVRC under a similar reasonable condition.
\end{list}

%Another reason for the above pessimistic bound is because that \cite{kohler2017} analyzed the sample complexity for sampling {\em with} replacement, whereas in practice sampling {\em without} replacement can potentially have much lower sample complexity. As a clear evidence, the sample complexity for sampling {\em with} replacement to achieve a certain accuracy can be {\em unbounded}, whereas this for sampling {\em without} replacement can only be as large as the total sample size. However, current theoretic analysis of sampling without replacement is not satisfactory. For example, \cite{Xu2017} used the same bound derived for sampling with and without replacement, which can be very loose for sampling without replacement.  In \cite{kohler2017}, although the experiment takes the total number of samples as the upper limit on sample complexity, the theoretical analysis focused only on the sampling {\em with} replacement, which did not precisely capture the experimental heuristic. Thus, the second open issue is to develop better bounds for sampling {\em without} replacement in order to provide more precise guidance for sub-sampled CR methods. 

 \renewcommand{\arraystretch}{1.3} 
\definecolor{LightCyan}{rgb}{0.88,1,1}
\begin{table*}[t] 
	\centering 
	\begin{tabular}{cllll} \toprule
		\multirow{2}{*}{Algorithms}&   &\multicolumn{1}{c}{Total}   &\phantom{a}    &\multicolumn{1}{c}{Total} \\  
		&   &\multicolumn{1}{c}{SHO}   & &\multicolumn{1}{c}{SO}  \\   \midrule 
		CR & \citep{Nesterov2006}   &  $\Oc(N\epsilon^{-3/2})$   &  &$\Oc(\epsilon^{-3/2}  )$    \\  \midrule 
		
		SCR &\citep{kohler2017}   &  $\Oc(\epsilon^{-5/2})$   &  &$\Oc(\epsilon^{-3/2}  )$    \\  \midrule 
		
		Inexact CR & \citep{Xu2017}  &  $\Oc(\epsilon^{-5/2})$   &  &$\Oc(\epsilon^{-3/2}  )$    \\  \midrule 
		
		SVRC(ZXG) & \citep{gu2018}  & $\Oc( N^{4/5} \epsilon^{-3/2}  )$  & &   $\Oc(  \epsilon^{-3/2}  )$  \\  \midrule \belowrulesepcolor{LightCyan}
		\rowcolor{LightCyan}
		SVRC   &(This Work)       & $\tilde{\Oc}( N^{2/3}\epsilon^{-3/2})$&   &   $\Oc(  \epsilon^{-3/2}  )$    \\  
		\aboverulesepcolor{LightCyan}  
		\bottomrule
	\end{tabular} 
	\vspace{2mm}
	\caption{Comparison of total Hessian sample complexity } \label{table_1}
\end{table*}
\footnotetext[1]{We note that SVRC(ZSG) does not need the objective function and its gradient to be Lipschitz but we adopt such assumptions.} 

In this paper, we address the aforementioned open issues, and our contributions are summarized as follows. 
 \vspace{-.1cm}
\subsection*{Our Contributions}
\vspace{-.1cm}
We propose a stochastic variance reduced cubic-regularized (SVRC) Newton's algorithm, which combines the variance reduced technique with concentration inequality under sub-sampling scheme.  We show that the computation of the full Hessian and gradient  can facilitate many steps of efficient inner-loop iteration as well as accurate approximation of  Hessian and gradient under high probability perspective. SVRC can be associated with two sampling schemes, respectively with and without replacement.

We establish the convergence guarantee of SVRC {\em with high probability} under the implementable inexact condition similar with $\norml{   \Hb_k - \nabla^2 F(\x_k)  } \leqslant C \norml{  \s_{k}}$. We show that the convergence of SVRC is at the same rate ($O(\epsilon^{-3/2})$) as the original CR \citep{Nesterov2006} or the other type of inexact-CR in \cite{Cartis2011a,Cartis2011b,kohler2017}. 

We then develop the bounds on the {\em total} Hessian sample complexity of SVRC. We show that SVRC achieves  $ \tilde{\Oc}(N^{2/3}\epsilon^{-3/2})$ Hessian sample complexity (where we use $\tilde{\Oc}$ to hide the dependence on log factors), which outperforms CR \citep{Nesterov2006} by an order of $O(N^{1/3})$ and outperform SCR \citep{kohler2017} in the regime of high accuracy requirement.  Furthermore, our proposed SVRC order-wise outperforms the algorithm SVRC(ZSG) \citep{gu2018} by an order of $O(N^{2/15})$, which is also a variance reduced cubic regularized method {\em concurrently proposed}. A detailed comparison among these algorithms are summarized in \Cref{table_1}.

%under both sampling without and with replacement 

% \footnote{Our paper appeared   at the same time on Arxiv as \cite{gu2018}. Since the same name (SVRC) is also adopted by the algorithm in \cite{gu2018}. Therefore, for clarification, we refer to the algorithm in \cite{gu2018} as SVRC(ZXG) with the authors' initial in the brackets.} by a factor of $O(N^{2/15})$. We summarize the detailed comparison below in \Cref{table_1}.  Furthermore, our Hessian sample complexity is based on the guarantee of high probability convergence instead of converging in expectation.

% After establishing the convergence guarantee,  our major goal is to meet the inexact conditions in \cref{new_inexact_condition} so that the convergence at the same rate as the original CR \cite{Nesterov2006} can be guaranteed. Thus, the remaining question lies in the statistical aspect. Towards this end, this paper makes the following three contributions.

We further provide an analysis for the case under sampling without replacement by developing a new concentration bound for sampling without replacement for random {\em matrices} by generalizing that for {\em scalar} random variables in \cite{bardenet2015}. Our result shows that sample replacement has lower  sample complexity than that of with replacement in each iteration.

\subsection*{Related Works}
 \textbf{Escaping saddle points:} Various algorithms have been developed to escape  strict saddle points and  converge to local minimum for nonconvex optimization. The first-order such algorithms include the gradient descent algorithm with random initialization \citep{lee2016} and with injection of random noise \citep{Ge2015,jinchi2017}. Various second-order algorithms were also proposed. In particular, \cite{YangT2017a,Yang2017b,JohnDuchi2016} proposed algorithms that exploit the negative curvature of Hessian to escape saddle points. The CR method as we describe below is another type of second-order algorithm that has been shown to escape strict saddle points.

\textbf{CR type of algorithms:} The CR method was shown in \cite{Nesterov2006} that converges  to a point that satisfies the first- and second-order optimality condition for nonconvex optimization. Its accelerated version was proposed in \cite{Nesterov2008} and the convergence rate was characterized for convex optimization. Several methods have been proposed to solve the cubic sub-problem in CR more efficiently. \cite{Cartis2011a} proposed to approximately solve the cubic sub-problem in Krylov space. \cite{Agarwal2017} proposed an alternative fast way to solve the sub-problem. \cite{Carmon2016} proposed a method based on gradient descent. 
  
\textbf{Inexact CR algorithms:} Various inexact approaches were proposed to approximate Hessian and gradient in order to reduce the computational complexity for CR type of  algorithms. In particular, \cite{Saeed2017} studied the inexact CR and accelerated CR for convex optimization, where the inexactness is fixed throughout the iterations. \cite{JinChi2017cubic} studied a similar inexact CR for nonconvex optimization. Alternatively, \cite{Cartis2011a,Cartis2011b} studied the inexact CR for nonconvex optimization, where the inexact condition is adaptive during the iterations. \cite{Wang2018inexact} established the convergence result of CR under a more reasonable inexact condition. \cite{Jiang2017} studied the adaptive inexact accelerated CR for convex optimization. In practice, sub-sampling is a very common approach to implement inexact algorithms. \cite{kohler2017} proposed a sub-sampling scheme that adaptively changes the sample complexity to guarantee the inexactness condition in \cite{Cartis2011a,Cartis2011b}. \cite{Xu2017} proposed uniform and nonuniform sub-sampling algorithms with fixed inexactness condition for nonconvex optimization.

\textbf{Stochastic variance reduced algorithms:} Stochastic variance reduced algorithms have been applied to various first-order algorithms (known as SVRG algorithms), and the convergence rate has been studied for convex functions in, e.g., \cite{ZhangTong2013,Lin_2014} and for nonconvex functions in, e.g., \cite{sashank2016}. \cite{gu2018} proposed a variance reduction version of CR. In this paper, we proposed another type of stochastic variance reduction to the second-order CR method to improve the state-of-art sample complexity result of approximating Hessian and gradient in probability perspective, and analyzed it in with and without replacement schemes. 

\textbf{Sampling without replacement:} The sampling without replacement scheme for first-order methods has been studied by various papers. \cite{Recht2012} and \cite{Shamir2016} studied  stochastic gradient descent under sampling without replacement for least square problems. \cite{Gzbalaban2015} provided convergence rate of the random reshuffling method. As for the sampling without replacement bounds, \cite{Hoeffding1963} showed that the bound for sampling with replacement also holds for sampling without replacement. \cite{Friedlander2012} provided deterministic bounds for without replacement sampling schemes for gradient approximations under certain assumptions. \cite{bardenet2015} provided tight concentration bounds for sampling without replacement for scalar random variables, while  bounds for random matrices remain unclear. We fill this gap, and provide a tight bound for random matrices under sampling without replacement in this paper.

%-----------------------------------------------------
% Proof Section
%-----------------------------------------------------

%\newpage
%\vspace{-.1cm}
\section{Stochastic Variance Reduction Scheme for Cubic Regularization}
%\vspace{-.1cm}
In this paper, we are interested in solving the  finite-sum problem given in \cref{eq: obj}, which is rewritten below.
\begin{align}
 \min_{\xb \in \RR^d} F(\xb) \triangleq \frac{1}{N} \sum_{i=1}^{N} f_i(\xb),   \label{eq: obj2}
\end{align}
 where the component functions $f_i, i=1, \ldots, N$ correspond to the loss of the $i$-th data samples, respectively, and is  nonconvex. More specifically, we adopt the following standard assumptions on the objective function in \cref{eq: obj2} throughout the paper 
 \begin{assum}\label{assum: obj}
 	The objective function in \cref{eq: obj2} satisfies
 	\begin{enumerate}[leftmargin=*,topsep=0pt,noitemsep]
 		\item  Function $F$ is bounded below, i.e., $\inf_{\xb \in \RR^d} F(\xb) > -\infty$;
 		\item  	For all component functions $f_i, i=1, \ldots, N$, the  function value $f_i$, the gradient $\nabla f_i$, and the Hessian $\nabla^2 f_i$ are  $L_0,L_1$ and $L_2$-Lipschitz,  respectively.
 	\end{enumerate}
 \end{assum}  
 
Classical first-order stochastic optimization methods such as stochastic gradient descent has a low sample complexity per-iteration \citep{Nemirovski_2009}. However, due to the variance of the stochastic gradients, the convergence rate is slow even with the incorporation of momentum \citep{Lan10-3,GhaLan15-1}. A popular approach to maintain the sample complexity yet achieve a faster convergence rate that is comparable to that of the full batch first-order methods is the stochastic variance reduction scheme \citep{ZhangTong2013,Lin_2014}. 

Motivated by the success of the variance reduction scheme in improving the sample complexity of first-order methods, we propose a {\em stochastic variance reduced cubic}-regularized Newton's method, and refer to it as SVRC. The detailed steps of SVRC are presented in \Cref{SVRC}. To briefly elaborate the notation in \Cref{SVRC}, we sequentially index the iterate variable $\x$ across all inner loops by $k$ for $k=0, 1, \ldots$, so that for each $\x_k$, the initial variable of its inner loop is indexed as $ \x_{\floor{k/m} \cdot m}$ (where $m$ is the number of iterations in each inner loop). For notational simplicity, we denote such an initial variable of each inner loop as $\tilde{ \x} $ and denote its corresponding full gradient and Hessian as $\tilde{\g}$ and $\tilde{\Hb}$, whenever there is no confusion.

\begin{algorithm}
  		\caption{SVRC} \label{SVRC}  
  		\begin{algorithmic}
  			\STATE {\bfseries Input:} $\x_0 \in \mathbb{R}^d$,  and  $\epsilon_1, m, M \in \mathbb{R^{+}}$.
  			\WHILE { $k$} 
  			\IF   {$k~\textrm{mod}~m = 0$  }
  			\STATE Set $\centering  \g_k= \nabla  F( \x_k) $,  $ \Hb_k = \nabla^2 F( \x_k)$, $  \widetilde{\g} = \g_k, \tilde{\x} = \x_k$ and $ \widetilde{\Hb} = \Hb_k $.
  			\ELSE
  			\STATE Sample index sets $\xi_{g}(k)$  and $ \xi_{H}(k)$ from $\{1,...,n\}$ uniformly at random.
  			 %\STATE   \textbf{Option I:}  Sampling  without replacement.
  			 %\STATE   \textbf{Option II:}  Sampling  with  replacement.
  			\STATE \textbf{Compute} 
  			\begin{align*}
  				\g_k \!&=\! \tfrac{1}{|\xi_{g}(k)|}  \big[ \!\textstyle\sum_{i \in \xi_{g}(k)} \big(\nabla f_i(\x_k)  \!-\! \nabla f_i(\tilde{\x})\big) \big]  \!+\! \tilde{\g}, \nonumber\\
  				\Hb_k \!&=\! \tfrac{1}{|\xi_{H}(k)|}  \big[\! \textstyle\sum_{i \in \xi_{H}(k)} (\nabla^2 f_i(\x_k)  \!- \! \nabla^2 f_i(\tilde{\x})) \!\big]  \!+\! \widetilde{\Hb}.
  			\end{align*}
  			\ENDIF
  			\STATE $\s_{k+1} = \argmin_{\s \in \RR^d}   \g_k^\top\s  + \frac{1}{2} \s^\top    \Hb_k\s  +\frac{M}{6}\norml{\s }^3$.
  			\STATE $\x_{k+1} =  \x_k + \s_{k+1}$.
  			\IF {$\max \{\norml{s_{k+1}}, \norml{s_{k}}\} \leqslant \epsilon_1$  }
  			\STATE return $x_{k+1}$
  			\ENDIF
  			\ENDWHILE
  		\end{algorithmic}
 \end{algorithm}
 
To elaborate the algorithm, SVRC calculates a full gradient $\widetilde{\g}$ and a full Hessian $\widetilde{\Hb}$ in every outer loop (i.e., for every $m$ iterations), which are further used to construct the stochastic variance reduced gradients $\g_k$ and Hessians $\Hb_k$ in the inner loops. 
Note that the index sets $\xi_{g}(k), \xi_{H}(k)$ for the sampled gradients and Hessians are generated by a random sampling scheme. More specifically, we consider the following two types of sampling schemes in this paper.

\textbf{Sampling with replacement:} For $k= 0, 1, \ldots$, each element of the index sets $\xi_g(k)$ and $\xi_H(k)$ is sampled uniformly at random from $\{1, \ldots, N\}$.

\textbf{Sampling without replacement:} For $k= 0, 1, \ldots$, the index sets $\xi_g(k)$ and $\xi_{H}(k)$ are sampled uniformly at random from all subsets of $\{1, \ldots, N\}$ with cardinality $|\xi_g(k)|$ and $|\xi_{H}(k)|$, respectively.

To elaborate, the sampling with replacement scheme may sample the same index multiple times within each mini-batch, whereas the sampling without replacement scheme samples each index at most once within each mini-batch. Therefore, the sampling without replacement scheme has a smaller variance compared to that of the sampling with replacement scheme. Consequently, these sampling schemes lead to inexact gradients and inexact Hessians with different guarantees to meet the inexactness criterion.

%The above two sampling schemes have very different statistics.  In particular,  sampling with replacement  may sample the same index multiple times, while  sampling without replacement  samples each index at most once  within each mini-batch. Thus,  sampling without replacement  has a smaller variance compared to that of  sampling with replacement. Consequently, these sampling schemes lead to inexact gradients and inexact Hessians with different guarantees to meet the inexactness criterions. 

%\textbf{SVRC$^+$ algorithm:} \label{SVRC_plus} Motivated by our theoretical analysis of SVRC (see \Cref{th:no_rep_totalsample} and \Cref{Total_Hessian_complexity_bound}), which suggests that SVRC outperforms SCR in terms of  Hessian sample complexity  in the high-accuracy regime, we propose the following variant algorithm   SVRC$^+$ as a more efficient implementation of SVRC. The idea is to first execute the algorithm as SCR, i.e., sample a mini-batch gradient and Hessian following the sample sizes suggested by SCR, until $\norml{\x_{k+1} - \x_{k} } < \eta$ is satisfied (where $\eta$ is a tuning parameter), and then execute the algorithm as SVRC as described in \Cref{SVRC}. In this way, SVRC$^+$ enjoys the efficiency of SCR and SVRC respectively in their advantageous regimes.

\section{Sample Complexity of SVRC}

In this section, we study the sample complexity of SVRC for achieving a second-order stationary point via three technical steps, each corresponding to one subsection below.

% To this end, it turns out that the inexact condition \citep{Wang2018inexact} on the estimated gradients and Hessians is not sufficient. Thus, we first propose a modified inexact condition and analyze the convergence to a second-order stationary point if SVRC satisfies such a condition. We then analyze the per-iteration and total sample complexity of SVRC in order to guarantee such a condition and hence converge to a second-order stationary point.
 
\subsection{Iteration Complexity under Modified Inexact Condition}

In order to analyze the sample complexity of SVRC for achieving a second-order stationary point, it turns out that the inexact condition \citep{Wang2018inexact} on the estimated gradients and Hessians is not sufficient. Thus, we propose a modified inexact condition below, and then analyze the convergence to a second-order stationary point if SVRC satisfies such a condition. 

%In this subsection, we modify the inexact condition \citep{Wang2018inexact} as follows.
%we study the iteration complexity of SVRC to converge to a second-order stationary point within $\epsilon$ accuracy. In particular, we perform our analysis of SVRC under an inexact condition \citep{Wang2018inexact} on the estimated gradients and Hessians, which can be satisfied via our variance-reduced sampling scheme. The inexact condition is stated as follows:
 \begin{assum} \label{assumption}
	The approximate Hessian $\mathbf{H}_k$ and approximate gradient $\mathbf{g}_k$ satisfy, for all $k = 0, \cdots$ ,
	\begin{align}
	\norml{   \Hb_k - \nabla^2 F(\x_k)  } \leqslant \alpha \max \left\{ \norml{  \s_{k}} ,\epsilon_1  \right\} \label{new_inexact_condition1} \\
	\norml{  \g_k - \nabla  F(\x_k)} \leqslant \beta \max \left\{ \norml{  \s_{k}}^2 ,\epsilon_1^2  \right\}  \label{new_inexact_condition2}
	\end{align} 
	where $\epsilon_1, \alpha$ and $\beta$ are universal positive constants.
\end{assum}
%We further discuss the sample complexity of our proposed variance reduction scheme to achieve the above inexact conditions in \Cref{sec: sample_com}. 
The inexact conditions in \cref{new_inexact_condition1,new_inexact_condition2} introduce a slack variable $\epsilon_1$ to avoid full batch sampling when $\norml{s_k}$ is very close to zero upon convergence. It turns out introduction of such a variable is essential for characterizing the total sample complexity of our proposed variance reduction scheme in \Cref{SVRC}. Furthermore, since \cref{new_inexact_condition1,new_inexact_condition2} are different from that in \citep{Wang2018inexact}, and hence require the convergence analysis if SVRC satisfies such conditions. The following theorem presents the iteration complexity analysis under the modified conditions. The technical proof in fact requires considerable extra effort than that in \citep{Wang2018inexact}. 

%Essentially, to obtain a second-order stationary point within certain accuracy, it is enough to compute an approximated Hessian under a mild inexact condition. 
%Fortunately, $\norml{\s_k}$ is larger than $\epsilon_1$ in most of cases and provide a more relax inexactness condition. 

 \begin{thm} \label{convergence_thm}
 \hspace{-2mm}	Suppose Assumption \ref{assum: obj} holds, and SVRC satisfies \ref{assumption}. Let
	\begin{align*} 
	 \tau \triangleq \min \Big\{ & \left(\frac{L+M}{2} + 2\beta + 2\alpha \right)^{-\frac{ 1}{2}},  \\
	 & \left( \frac{M+2L}{2}+ 2\alpha \right)^{-1}  \Big\},
 	\end{align*} 
 	set
	\begin{align}
 		\epsilon_1 = \tau \sqrt{\epsilon}, \label{cov_10}
 	\end{align}
 	and properly choose $M, \alpha$ and $\beta \in \mathbb{R}$ such that
 	\begin{align}
 	\gamma \triangleq  \left(\frac{3M-2L_2}{24} - \frac{5}{2}\beta -\frac{5}{4}\alpha \right) > 0.
 	\end{align}
 	Then, the SVRC algorithm outputs an $\epsilon$-approximate second-order stationary point, i.e.,
 	\begin{align}
 	\norml{\nabla f(\x_{k+1})} \leqslant \epsilon \quad \text{ and }  \quad  \nabla^2 f(\x_{k+1}) \succcurlyeq - \epsilon \mathbf{I}
 	\end{align}
 	within at most   $  k  = O\left(  \epsilon^{-3/2}\right)$ number of iterations. 
 	Moreover, the following inequality holds
 	\begin{align}
 	\sum_{i=1}^{k+1}    \norml{\s_{i} }^3   &\leqslant  C, \label{vector_norm_bound_7}
 	\end{align} 
 	where $C \triangleq ( {f(\x_0)  - f^*    +\left(2\beta +  {\alpha} + 2\gamma \right) \epsilon_1^3})/{\gamma}$.
 \end{thm}

As stated in \Cref{convergence_thm}, SVRC outputs an $\epsilon$-approximate second-order stationary point with  $ k  = O\left(  \epsilon^{-3/2}\right) $. Such an iteration complexity matches the state-of-art result and is the best result that one can expect on nonconvex optimization.

% \subsection{Sample Complexity of SVRC} \label{sec: sample_com}
%% In this section, we study the sample complexity of SVRC for achieving a second-order stationary point. 
%As the computation complexity of Hessian dominates that of gradient, we focus on the total Hessian sample complexity of SVRC. We first bound the per-iteration complexity in \Cref{with_sec_1}, and bound the total Hessian complexity in \Cref{with_sec_2}.

\subsection{Per-iteration Sample Complexity} \label{with_sec_1}

In this subsection, we bound the per-iteration sample complexity in order for SVRC (under sampling with replacement) to satisfy the inexact conditions in \cref{new_inexact_condition1,new_inexact_condition2}. We apply Bernstein's inequality and obtain the following theorem.
%leads to the following propositions on the sample complexity for the inexact criterions under the sampling with replacement.
\begin{thm}\label{With_Replacement_Gradient}
Let Assumption  \ref{assum: obj}   hold. Consider SVRC under the sampling with replacement scheme. Then, the sub-sampled mini-batch of gradients $\g_k, k = 0, 1, \ldots$ satisfies \Cref{assumption} with probability at least $1 - \zeta$ provided that 
	\begin{align*}
		|\xi_{g}(k)| &\geqslant   \bigg(    \frac{ 8L_1^2 }{ \beta^2 \max \{\norml{\s_k}^4, \epsilon_1^4\}} \norml{\x_{k} - \tilde{\x}}^2  \\
		&\hspace{-5mm} + \frac{ 4L_1 }{3  \beta \max \{\norml{\s_k}^2, \epsilon_1^2\}} \norml{\x_{k} - \tilde{\x}}    \bigg) \log \left(\frac{ 2(d+1)}{\zeta}\right), \numberthis  \label{eq:svrc_rep_g}
%		 \Oc \left(    \frac{ \norml{\x_{k} - \tilde{\x}}^2 }{   \max \{\norml{\s_k}^4, \epsilon_1^4\}}    + \frac{  \norml{\x_{k} - \tilde{\x}} }{ \max \{\norml{\s_k}^2, \epsilon_1^2\}}     \right),
	\end{align*} 
Furthermore, the sub-sampled mini-batch of Hessians $\Hb_k, k = 0, 1, \ldots$ of SVRC satisfies \Cref{assumption} with probability at least $1 - \zeta$ provided that 
	\begin{align*}\label{eq:svrc_rep_h}
	\hspace{-2mm}	|\xi_{H}(k)|  &\geqslant  \bigg(    \frac{ 8L_2^2 }{ \alpha^2 \max \{\norml{\s_k}^2, \epsilon_1^2\}} \norml{\x_{k} - \tilde{\x}}^2  \\
	& + \frac{ 4L_2 }{3  \alpha \max \{\norml{\s_k}, \epsilon_1\}} \norml{\x_{k} - \tilde{\x}}    \bigg) \log \left(\frac{ 4d}{\zeta}\right)   \numberthis.
	\end{align*} 
\end{thm} 
We next compare the per-iteration Hessian sample complexity of SVRC under the sampling with replacement scheme (\cref{eq:svrc_rep_h}) with that of SCR under the same sampling scheme developed in \cite{kohler2017}, which is rewritten below
\begin{align} \label{subsampling_sample_size_1}  
	|\xi_{H}(k)| \geqslant \Oc  \left(\tfrac{1}{\norml{\s_{k+1}}^2}\right).
\end{align}
To compare, our \Cref{With_Replacement_Gradient} requires a Hessian sample complexity of roughly the order 
\begin{align} \label{with_per_bound} 
	|\xi_{H}(k)| \geqslant \Oc  \left(\tfrac{\norml{\x_{k} - \tilde{\x}}^2}{\norml{\s_{k }}^2}\right).
\end{align}

 It can be seen that the sample complexity bounds for SVRC in \cref{with_per_bound}  have an additional term $\norml{\x_k - \tilde{\x} }^2$ in the numerators comparing to their corresponding bound for SCR in \cref{subsampling_sample_size_1}. Intuitively, $\norml{\x_{k} - \tilde{\x}}  \rightarrow 0$ as the algorithm converges, and thus our variance reduction scheme requires a lower sample complexity than the stochastic sampling in SCR.
   
\subsection{Total Sample Complexity of SVRC} \label{with_sec_2} 
\Cref{With_Replacement_Gradient} provides the sample complexity per iteration (each iteration in SVRC inner loop). We next provide our   result on the sample complexity over the running process of SVRC, which is a key factor that impacts the computational complexity of SVRC.  

\begin{thm} \label{Total_Hessian_complexity_bound}
	Let Assumptions \ref{assum: obj}  hold. For a given $\epsilon$ and $\delta$,  then SVRC under the sampling with replacement scheme outputs an point $\x_{k+1}$ such that satisfies $\norm{\nabla F(\xb_{k+1})} \leqslant  \epsilon$ and $\nabla^2F(\xb_{k+1}) \succcurlyeq -\epsilon \bm{I}$ with probability at least $1 - \delta$, and  the total Hessian sample complexity of SVRC is bounded by
    \begin{align*}
    	  \sum_{i=1}^{K} | \xi_{H}(i)|  &\leqslant    \frac{ C N^{2/3} }{\epsilon^{ 3/2}}\log \left(\frac{ 8d}{\epsilon  \delta  }\right). 
    \end{align*}
\end{thm}

We next compare the total   Hessian sample complexity of SVRC with that of other CR-type algorithms, which are given below.
\begin{align}
\text{SVRC:}  &\quad  \sum_{i=1}^{K} | \xi_{H}(i)|    = \tilde{\Oc} \left(\frac{ N^{2/3}}{\epsilon^{3/2}}\right), \label{SVRC_Hessian_complexity} \\
	\text{SVRC (ZXG):}  &\quad  \sum_{i=1}^{K} | \xi_{H}(i)|    = \Oc \left(\frac{ N^{4/5}}{\epsilon^{3/2}}\right), \label{gu_bound} \\
		\text{CR:}	 &\quad	  \sum_{i=1}^{K} | \xi_{H}(i)|  \leqslant   \Oc \left(\frac{ N }{\epsilon^{3/2}}\right), 
		\label{CR_Hessian_complexity} \\
		\text{SCR:}   &\quad \sum_{i=1}^{K} | \xi_{H}(i)| \leqslant  \Oc \left(\frac{1}{\epsilon^{5/2}}\right) \label{SCR_Hessian_complexity}.
\end{align}

Comparing \cref{SVRC_Hessian_complexity,gu_bound,CR_Hessian_complexity}. Clearly, our SVRC has lower total sample complexity than CR and SVRC(ZXG) by an order of $O(N^{1/3})$ and $O(N^{2/15})$, respectively. Therefore, our stochastic variance reduction scheme is sample efficient when applied to CR type of methods.  
%  This is consistent with the comparison for sampling without replacement.
 Also, comparing the sample complexity of the two subsampled algorithms in \cref{SVRC_Hessian_complexity,SCR_Hessian_complexity}, we observe that SVRC enjoys a lower-order complexity bound than SCR if $\epsilon =o(  N^{-2/3})$, and hence performs better in the high accuracy regime.

%Furthermore, we compare the Hessian sample complexity for SVRC under sampling with replacement in \Cref{Total_Hessian_complexity_bound}  and under sampling without replacement in \Cref{th:no_rep_totalsample}. It can be observed that sampling without replacement saves the sample size with an order of $O(N^{1/44})$.

%Here, we are also interested in the sample size in the last outer loop, which we characterize in the following theorem.
%\begin{thm} \label{last_step_inner_loop_constant_samples}
%Let Assumptions \ref{assum: obj} and \ref{assum: bounded_gH} hold. Consider SVRC under sampling with replacement. Assume the algorithm terminates at iteration $k_0 = K$. Then for each iteration in the last outer loop before the algorithm terminates, i.e., for $  \floor{K/m}+1 \leqslant k \leqslant K-1$, the required sample size satisfies
%		\begin{align} 
%			 |\xi_{g}(k)| \geqslant C\frac{ m^2 }{\epsilon},  \quad |\xi_{H}(k)| \geqslant Cm^2.   
%		\end{align}
%\end{thm} 
% \begin{remark}
%We choose $m^\star  = C  N^{1/3}$, which achieves an optimal trade-off between the number of full Hessian computations and the sample size of each inner loop iteration.
%%by minimizing the over all Hessian sample complexity which derives form \cref{with_replacment_m_satr}.
% \end{remark}
%
%The above theorem suggests that in the worst case $k_0 = K$, the iterations in the last outer loop can use a constant sample size, which is in contrast to SCR, whose sample size continues to increase before the algorithm stops as suggested by \cref{subsampling_sample_size_1}. 

\section{SVRC under Sampling without Replacement Scheme}
In this section, we explore the sample complexity of SVRC under the sampling without replacement scheme, which is commonly used in practice. 

%\textbf{Sampling without replacement:} 
%For $k= 0, 1, \ldots$, the index sets $\xi_g(k)$ and $\xi_{H}(k)$ are sampled uniformly at random from all subsets of $\{1, \ldots, N\}$ with cardinality $|\xi_g(k)|$ and $|\xi_{H}(k)|$, respectively.
%
%To elaborate, the sampling with replacement scheme may sample the same index multiple times  within each mini-batch, whereas the sampling without replacement scheme samples each index at most once  within each mini-batch. Therefore, the sampling without replacement scheme has a smaller variance compared to that of the sampling with replacement scheme. Consequently, these sampling schemes lead to inexact gradients and inexact Hessians with different guarantees to meet the inexactness criterion.

To this end, we first develop some technical concentration inequalities in the next subsection.

\subsection{Concentration Inequality under Sampling without Replacement}

The statistics of sampling without replacement is very different and more stable than that of sampling with replacement. However, theoretical analysis of sampling {\em without} replacement turns out to be very difficult. A common approach is to apply the concentration bound for sampling with replacement, which also holds for sampling without replacement \citep{Tropp2012}. However, such analysis can be too loose to capture the essence of the scheme of sampling without replacement. For example, the sample complexity for sampling {\em with} replacement to achieve a certain accuracy can be {\em unbounded}, whereas sampling {\em without} replacement can at most sample the total sample size. 

Thus, in order to develop a tight sample complexity bound for SVRC under sampling without replacement, we first leverage a recently developed Hoeffding-type of concentration inequality for sampling {\em without} replacement \citep{bardenet2015}. There, the result is applicable only for scalar random variables, whereas our analysis here needs to deal with  sub-sampled gradients and Hessians, which are vectors and matrices. This motivates us to first establish the matrix version of the Hoeffding-Serfling inequality. Such a concentration bound can be of independent interest in various other domains. The proof turns out to be very involved and is provided in the supplementary materials.  
 
 \begin{thm} \label{Matrix_Hoeffding_Serfling_Inequality}
 	Let $\mathcal{X}:= \{\A_1, \cdots, \A_N\}$ be a collection of real-valued matrices in $\RR^{d_1\times d_2}$ with bounded spectral norm, i.e.,  $\norml{\A_i} \leqslant\sigma$ for all $i = 1, \ldots, N$ and some $\sigma > 0$. 
 	Let $\X_1, \cdots, \X_n$ be $n<N$ samples from $\mathcal{X}$ under the sampling without replacement. Denote $\mu := \frac{1}{N} \sum_{i = 1}^{N} \A_i$. Then,  for any $\epsilon > 0$, the  following bound holds.
 	\begin{align*}
 	P & \bigg(\bigg\|  \frac{1}{n}\sum_{i=1}^{n}  \X_i -   \mu\bigg\|  \geqslant   \epsilon \bigg)  \\
 	&\quad   \quad    \leqslant    2(d_1 + d_2) \exp \bigg( -  \frac{n \epsilon^2}{8 \sigma^2 (1+1/n) (1- n/N)}\bigg).
 	\end{align*}
 \end{thm}
 
To further understand the above theorem, consider symmetric random matrix $\X_i \in \RR^{d\times d}$. Suppose we want $ \norml{  \frac{1}{n}\sum_{i=1}^{n}  \X_i -   \mu}  \leqslant   \epsilon  $ to hold with probability $1- \zeta$. Then the above theorem requires  the sample size to satisfy
\begin{align} \label{sample_size_without_raplacement}
	n_{w} \geqslant \bigg(\frac{1}{N} + \frac{\epsilon^2}{16 \sigma^2 \log(4d)/ \zeta)}\bigg)^{-1}.
\end{align} 

We consider two regimes to understand the bound in \cref{sample_size_without_raplacement}. (a) Low accuracy regime: Suppose $\epsilon$ is large enough so that the second term in \cref{sample_size_without_raplacement} dominates. In this case, we roughly have $n_w \geqslant \frac{16 \sigma^2 \log(4d/\zeta)}{\epsilon^2}$, which has the same order as the suggested sample size by the matrix version of the Hoeffding inequality for sampling {\em with} replacement given below
\begin{align} \label{sample_size_with_raplacement}
n_b \geqslant \frac{8 \sigma^2 \log(2d/\zeta)}{\epsilon^2}.
\end{align}
Thus, the sample size is approximately the same for sampling with and without replacement to achieve a low accuracy concentration. (b) High accurary regime: Suppose $\epsilon$ is small enough so that the first term  in \cref{sample_size_without_raplacement} dominates. Hence, \cref{sample_size_without_raplacement} roughly reduces to $n_w \geqslant N$, whereas the matrix version of the Hoeffding bound in \cref{sample_size_with_raplacement} for sampling with replacement requires infinite samples as $\epsilon \to 0$. Thus, the sample size is highly different for sampling with and without replacement to achieve a high accuracy concentration.

\vspace{-.1cm}
\subsection{Per-iteration Sample Complexity}
\vspace{-.1cm}
We apply \Cref{Matrix_Hoeffding_Serfling_Inequality} to analyze the sample complexity of SVRC under sampling without replacement. Our next theorem characterizes the sample size needed for SVRC in order to satisfy the inexact condition in \Cref{assumption}.  
 \begin{thm} \label{SVRC_Without_Replacement}
 Let Assumption  \ref{assum: obj} hold. Consider SVRC under sampling without replacement. The sub-sampled mini-batches of gradients $\g_k, k = 0, 1, \ldots$ satisfy \cref{eq: inexact_g} with probability at least $1 - \zeta$ provided that 
 	\begin{align}\label{eq:svrc_nrep_g}
 	|\xi_{g}(k)| \!\geqslant\! \bigg(\frac{1}{N} + \frac{\beta^2 \max \{\norml{  \s_{k}}^4, \epsilon_1^4\}}{64 L_1^2 \norml{\x_{k} - \tilde{\x}}^2 \log(2(d+1)/\zeta)}\bigg)^{-1}\!\!\!,
 	\end{align} 

Furthermore, the sub-sampled mini-batches of Hessians $\Hb_k, k = 0, 1, \ldots$ satisfy \cref{eq: inexact_H} with probability at least $1 - \zeta$ provided that 
 	\begin{align}\label{eq:svrc_nrep_h}
 	|\xi_{H}(k)| \geqslant\left( {\frac{1}{N} + \frac{\alpha^2 \max \{\norml{\s_k}^2, \epsilon_1^2\}}{64 L_2^2 \norml{\x_{k} - \tilde{\x}}^2 \log(4d/\zeta)}}\right)^{-1}.
 	\end{align} 
 \end{thm}

In order to further understand the sample complexity in \Cref{SVRC_Without_Replacement} and what improvement that SVRC makes in terms of sample complexity compared to the SCR algorithm in \cite{kohler2017}, we next characterize the corresponding sample complexity for SCR under sampling without replacement below. (We note that the sample complexity for SCR under sampling with replacement was provided in \cite{kohler2017}.)
%\subsubsection{Sample Complexity of SCR under Sampling Without Replacement}
\begin{proposition} \label{SCR_Without_Replacement}
Let Assumptions \ref{assum: obj} hold. Consider the SCR algorithm in \cite{kohler2017} under sampling without replacement. The sub-sampled mini-batch of gradients $\g_k, k = 0, 1, \ldots$ satisfies \cref{eq: inexact_g} with probability at least $1 - \zeta$ provided that for all $k$
	\begin{align}\label{eq:scr_nrep_g}
	|\xi_{g}(k)| \geqslant \bigg(\frac{1}{N} + \frac{C_1^2\norml{\x_{k+1} - \x_{k}}^4}{64 L_0^2   \log(2(d+1)/\zeta)}\bigg)^{-1}. 
	\end{align}
Furthermore, the sub-sampled mini-batch of Hessians $\Hb_k, k = 0, 1, \ldots$ satisfies \cref{eq: inexact_H} with probability at least $1 - \zeta$ provided that for all $k$
	\begin{align}\label{eq:scr_nrep_h}
	 |\xi_{H}(k)| \geqslant \bigg(\frac{1}{N} + \frac{C_2^2\norml{\x_{k+1} - \x_{k}}^2}{64 L_1^2   \log(4d/\zeta)}\bigg)^{-1}.
	\end{align}
\end{proposition}

To compare the sample complexity for SVRC in \Cref{SVRC_Without_Replacement} and SCR in \Cref{SCR_Without_Replacement}, we take the sample complexity for mini-batch of gradients as an example. Comparing \cref{eq:svrc_nrep_g} and \cref{eq:scr_nrep_g}, the second term in the denominator in \cref{eq:svrc_nrep_g} is additionally divided by $\norml{\x_k - \tilde{\x}}^2$, which converges to zero as the algorithms converge. Thus, $\norml{\x_{k+1} - \x_k}^2$ in \cref{eq:scr_nrep_g} converges to zero much faster than $ \frac{\norml{\x_{k+1} - \x_k}^4}{\norml{\x_k - \tilde{\x}}^2}$ in \cref{eq:svrc_nrep_g}, so that the term $1/N$ dominates the denominator and results in the sample size close to the number of total samples much earlier in the iteration of  SCR than SVRC.  

We also note that \Cref{SCR_Without_Replacement} shows that as SCR approaches the convergence, the sample size goes to the total number of samples with technical rigor, whereas such a fact was only intuitively discussed in \cite{kohler2017}.
\vspace{-.1cm}
\subsection{Total Sample Complexity}
\vspace{-.1cm}
%We next characterize a few important properties regarding the convergence of the SVRC algorithm. Recall that $\{\x_{k} \}$ and $\{\tilde{ \x} \}$ are the sequences generated by SVRC in the inner loops and outer loops, respectively.
% 
%\begin{proposition} \label{converge_result_pass_point}
%Let Assumptions \ref{assum: obj}, \ref{assum: bounded_gH}, and \ref{assumption} hold. Then the iterates of SVRC satisfy
%	 \begin{align}
%	 \lim\limits_{ k \rightarrow \infty }\norml{\x_k -  \tilde{ \x}} = 0.
%	 \end{align}
%\end{proposition}
%We note that the above result holds for both sampling with and without replacement.
We next characterize the total Hessian sample complexity of SVRC under sampling without replacement.
 
\begin{thm}\label{th:no_rep_totalsample}
	Let Assumptions \ref{assum: obj}  hold. For a given $\epsilon$ and $\delta$ then SVRC under sampling without replacement outputs an point $\x_{k+1}$ such that satisfies $\norm{\nabla F(\xb_{k+1})} \leqslant  \epsilon$ and $\nabla^2F(\xb_{k+1}) \succcurlyeq -\epsilon \bm{I}$ with probability at least $1 - \delta$. Then the total sample complexity for Hessian used in SVRC is bounded by 
	\begin{align} \label{eq:no_rep_totalsample}
	 \sum_{i=0}^{k} |\xi_{H}(k)|  \leqslant  \frac{ C N^{2/3} }{\epsilon^{ 3/2}}\log \left(\frac{ 8d}{\epsilon  \delta  }\right).
	\end{align}
\end{thm} 

In this theorem, we show that  total sample complexity of SVRC under sampling without replacement is at least as good as SVRC under sampling with replacement. And the comparison of this bound with other bound follows similarly as we discuss in \Cref{with_sec_2}.

%It is clear that as the algorithm converges, more samples are typically needed to approximate gradient and Hessian more accurately in order to satisfy the optimality criterion. Thus, the sample size in the last outer loop provides useful information for understanding the worst-case sample complexity. 
%
%\begin{thm}\label{th:no_rep_laststep}
% 	Let Assumptions \ref{assum: obj} and \ref{assum: bounded_gH} hold. Consider SVRC under sampling without replacement. Assume the algorithm terminates at iteration $ k_0 = K$. Then in each iteration of the last outer loop before the algorithm stops, i.e., for $  \floor{K/m}+1 \leqslant k \leqslant K-1$, the required sample size satisfies
% 	\begin{align} 
% 	|\xi_{H}(k)| &\geqslant  \left(\frac{1}{N} +   \frac{C}{m^2}\right)^{-1},  \label{eq:no_rep_laststep_h} \\
% 	|\xi_{g}(k)| &\geqslant  \left(\frac{1}{N} +   \frac{C \epsilon}{m^2}\right)^{-1}.
% 	\end{align}
% 	where $C$ is a constant.
% \end{thm} 
%
%The above theorem suggests that in the worst case $k_0 = K$, the iterations in the last outer loop can use constant sample size for Hessian as in \cref{eq:no_rep_laststep_h}. This is in contrast to SCR, the sample size of which keeps on increasing as can be seen in \cref{eq:scr_nrep_h}, because $\|\s_k\|=\|\x_{k+1}-\x_k\|$ continues to descrease before SCR stops. In fact, our experiment implements a SVRC algorithm with a constant sample size for all loops for SVRC and demonstrates satisfactory performance.

%\input{latex_subFiles/Experiment.tex}
 
\section{Discussion}
 \textbf{Storage Issue:} The proposed algorithm involves the storage of a Hessian, which requires $O(d^2)$ space for storage. In this perspective,  the proposed algorithm can be directly applied for solving small or medium scale machine learning problems. As for large scale problems, using  PCA to store the main component of Hessian can be a possible solution. 
 
 \textbf{With and Without replacement:}
 We show that the total sample complexity of SVRC under sampling without replacement is at least as good as SVRC under sampling with replacement.  Actually, if we compare the per iteration complexity of the two, i.e., we compare \Cref{SVRC_Without_Replacement} with \Cref{With_Replacement_Gradient}, the without replacement scheme has a better complexity than that with replacement in each iteration since there is  a $1/N$ term in the denominator on the bound for the scheme without replacement. This does suggest the same total sample complexity for the two schemes is likely due to the technicality issue.
 
 % Thus, whether there is some change to prove a overall better result than the result we stated in \Cref{th:no_rep_totalsample} is remained to be an open question.
 
\section{Conclusion} 
In this paper, we proposed a stochastic variance-reduced cubic regularization method. We characterized the per iteration sample complexity for Hessian and gradient that guarantees convergence of SVRC to a second-order optimality condition, under both sampling with and without replacement. We also developed the total sample size for Hessian. Our theoretic results imply that SVRC outperforms the state-of-art result by an factor of $O(N^{2/15})$. Moreover, Our study demonstrates that variance reduction can bring substantial advantage in sample size as well as computational complexity for second-order algorithms, along which direction we plan to explore further in the future.

\bibliographystyle{apalike}  
\bibliography{ref}

\newpage 
\onecolumn
\appendix
\noindent {\Large \textbf{Supplementary Materials}}

\section{Proof of Convergence}
\subsection{Lemmas}
In this subsection, we introduce two useful lemmas, which will be used in the proof of convergence.
\begin{lemma}[\cite{Nesterov2006}, Lemma 1] \label{Hessian_square_bound}
	Let the Hessian $\nabla^2 f(\cdot)$ of the function   $f(\cdot)$ be $L $-Lipschitz continuous with $L  > 0$. Then, for any $\x, \mathbf{y} \in \mathbb{R}^d$, we have
	\begin{align}
	\norml{\nabla f( \mathbf{y} ) - \nabla f(\x) - \nabla^2 f(\x )(\mathbf{y} - \x ) }  &\leqslant  \frac{L }{2} \norml{\mathbf{y}  - \x}^2,  \label{gradient_bound}\\
	\biggl| f(\mathbf{y}) - f(\x) - \nabla f(\x)^T(\mathbf{y}  - \x)     - \frac{1}{2} (\mathbf{y}  - \x)^T  \nabla^2 f(\x )(\mathbf{y} - \x ) \biggr|
	&\leqslant \frac{L }{6}  \norml{\mathbf{y}  - \x}^3. \label{function_vaule}
	\end{align}
\end{lemma}

\begin{lemma}[\cite{Wang2018inexact}, Lemma 3] \label{subcubic}
	Let $M \in \mathbb{R}, \mathbf{g} \in \mathbb{R}^d, \mathbf{H} \in  \mathbb{S}^{d \times d}$, and 
	\begin{align}
	\s  = \argmin_{\ub \in \mathbb{R}^d} \mathbf{g}^\top \ub + \frac{1}{2} \ub^\top \mathbf{H} \ub  + \frac{M}{6} \norml{\ub}^3. \label{opt}
	\end{align}
	Then, the following statements  hold:
	\begin{align}
	\mathbf{g}  +      \mathbf{H} \s   + \frac{M}{2} \norml{\s } \s  &= \mathbf{0}, \label{opt_1} \\
	\mathbf{H} + \frac{M}{2} \norml{\s } \mathbf{I}    &\succcurlyeq \mathbf{0}, \label{opt_2} \\
	\mathbf{g}^\top \s  + \frac{1}{2}    \s^\top \mathbf{H} \s   + \frac{M}{6} \norml{\s }^3 &\leqslant - \frac{M}{12}\norml{\s }^3. \label{opt_3}
	\end{align}
\end{lemma} 

\subsection{Proof of \Cref{convergence_thm}}  
\begin{proof}
	Since $\nabla^2 f(\x)$ is $L_2$-Lipschitz, thus we have
		\begin{align*}
	f(\x_{k+1} ) - f(\x_{k}) &\numleqslant{i}     \nabla f(\x_k)^\top\s_{k+1}     + \frac{1}{2} \s_{k+1}^\top \nabla f(\x_k)\s_{k+1}  + \frac{L_2}{6} \norml{\s_{k+1}}^3 \\
	& \leqslant \mathbf{g}_k^\top\s_{k+1}     + \frac{1}{2} \s_{k+1}^\top \mathbf{H}_k  \s_{k+1}+ \frac{M}{6} \norml{\s_{k+1}}^3  +  (\nabla f(\x_k) - \mathbf{g}_k)^\top\s_{k+1} \\
	&\qquad + \frac{L_2-M}{6} \norml{\s_{k+1}}^3  +\frac{1}{2} \s_{k+1}^\top (\nabla^2 f(\x_k) - \mathbf{H}_k)\s_{k+1}  \\
	&\numleqslant{ii} - \frac{3M-2L_2}{12}\norml{\s_{k+1} }^3 + (\nabla f(\x_k) - \mathbf{g}_k)^\top\s_{k+1}  + \frac{1}{2} \s_{k+1}^\top (\nabla f(\x_k) - \mathbf{H}_k)\s_{k+1} \numberthis \label{cov_1}
	\end{align*}  
	where (i) follows from \Cref{Hessian_square_bound} with $\mathbf{y} = \x_{k+1}, \x = \x_{k}$ and $\s_{k+1} =\x_{k+1} - \x_{k}$, (ii) follows from \cref{opt_3} in \Cref{subcubic} with $\mathbf{g} = \mathbf{g}_k,  \mathbf{H} =\mathbf{H}_k $ and $\s = \s_{k+1}$, (iii) follows from \Cref{assumption},

	Next, we bound the terms $(\nabla f(\x_k) - \mathbf{g}_k)^\top\s_{k+1}$ and $  \s_{k+1}^\top (\nabla f(\x_k) - \mathbf{H}_k)\s_{k+1}$. For the first term, we have that
	\begin{align*}
	(\nabla f(\x_k) - \mathbf{g}_k)^\top\s_{k+1} &\leqslant \norml{\nabla f(\x_k) - \mathbf{g}_k}\norm{\s_{k+1}} \numleqslant{i} \beta \left(\norm{\s_{k }}^2 + \epsilon_1^2 \right)\norm{\s_{k+1}} = \beta \left(\norm{\s_{k }}^2\norm{\s_{k+1}} + \epsilon_1^2\norm{\s_{k+1}} \right)  \\
	& \numleqslant{ii}  \beta \left(\norm{\s_{k }}^3 + \norm{\s_{k+1}}^3 + \epsilon_1^3 + \norm{\s_{k+1}}^3 \right) =  \beta \left( \norm{\s_{k }}^3 + 2 \norm{\s_{k+1}}^3 + \epsilon_1^3   \right) , \numberthis \label{cov_2}
	\end{align*}
	where (i) follows from \Cref{assumption}, which gives that $\norml{  \g_k - \nabla  F(\x_k)} \leqslant \beta \max \left\{ \norml{  \s_{k}}^2 ,\epsilon_1^2  \right\}$, and (ii) follows from  the inequality that for $a,b \in \mathbb{R}^+$, $a^2b \leqslant a^3 + b^3$, which can be verified by checking the  cases with $a < b$  and $a \geqslant b$, respectively. Similarly, we obtain that
	\begin{align*}
	\s_{k+1}^\top (\nabla f(\x_k) - \mathbf{H}_k)\s_{k+1} &\leqslant \norml{\nabla^2 f(\x_k) - \mathbf{H}_k}\norm{\s_{k+1}}^2 \numleqslant{i} \alpha \left(\norm{\s_{k }}  + \epsilon_1  \right)\norm{\s_{k+1}}^2 = \alpha \left(\norm{\s_{k }} \norm{\s_{k+1}}^2 + \epsilon_1 \norm{\s_{k+1}}^2 \right)  \\
	& \numleqslant{ii}  \alpha \left(\norm{\s_{k }}^3 + \norm{\s_{k+1}}^3 + \epsilon_1^3 + \norm{\s_{k+1}}^3 \right) = \alpha \left(\norm{\s_{k }}^3 + 2\norm{\s_{k+1}}^3 + \epsilon_1^3  \right), \numberthis \label{cov_3}
	\end{align*} 
	where (i) follows from \Cref{assumption}, which gives that $\norml{   \Hb_k - \nabla^2 F(\x_k)  } \leqslant \alpha \max \left\{ \norml{  \s_{k}} ,\epsilon_1  \right\}$, and (ii) follows from  the inequality that for $a,b \in \mathbb{R}^+$, $a^2b \leqslant a^3 + b^3$.
	
	Plugging \cref{cov_2,cov_3} into \cref{cov_1} yields  
	\begin{align*}
	f(\x_{k+1} ) - f(\x_{k}) &\leqslant  - \frac{3M-2L_2}{12}\norml{\s_{k+1} }^3  +  \beta \left(\norm{\s_{k }}^3 + 2\norm{\s_{k+1}}^3 + \epsilon_1^3  \right) + \frac{\alpha}{2}\left(\norm{\s_{k }}^3 + 2\norm{\s_{k+1}}^3 + \epsilon_1^3  \right) \\
	&=  - \left(\frac{3M-2L_2}{12} - 2\beta -\alpha \right)\norml{\s_{k+1} }^3 + \left(\beta + \frac{\alpha}{2}\right)\norm{\s_{k}}^3 + \left(\beta + \frac{\alpha}{2}\right)\epsilon_1^3 \numberthis \label{cov_4}
	\end{align*}
	
	Summing \Cref{cov_4} for $0$ to $k$, we obtain
	\begin{align*}
	f(\x_{k+1}) - f(\x_0) &\leqslant  - \left(\frac{3M-2L_2}{12} - 2\beta -\alpha \right) \sum_{i=1}^{k+1}\norml{\s_{i} }^3 + \left(\beta + \frac{\alpha}{2}\right) \sum_{i=0}^{k} \norm{\s_{i}}^3 + \left(\beta + \frac{\alpha}{2}\right) \sum_{i=0}^{k}\epsilon_1^3 \\
	&\leqslant   - \left(\frac{3M-2L_2}{12} - 2\beta -\alpha \right) \sum_{i=1}^{k+1}\norml{\s_{i} }^3 + \left(\beta + \frac{\alpha}{2}\right) \sum_{i=0}^{k+1} \norm{\s_{i}}^3 + \left(\beta + \frac{\alpha}{2}\right) \sum_{i=0}^{k}\epsilon_1^3 \\
	&\leqslant   - \left(\frac{3M-2L_2}{12} - 3\beta -\frac{3}{2}\alpha \right) \sum_{i=1}^{k+1}\norml{\s_{i} }^3 +    \left(\beta + \frac{\alpha}{2}\right) \norm{\s_{0}}^3 + \left(\beta + \frac{\alpha}{2}\right) \sum_{i=0}^{k}\epsilon_1^3 , \numberthis \label{cov_6}
	\end{align*}  
	We next note that
	\begin{align}
	\sum_{i=1}^{k+1}\norml{\s_{i} }^3 = \frac{1}{2} \left(\sum_{i=1}^{k+1}\norml{\s_{i} }^3 + \sum_{i=1}^{k+1}\norml{\s_{i} }^3 \right) =  \frac{1}{2} \left(\sum_{i=1}^{k+1}\norml{\s_{i} }^3 + \sum_{i=0}^{k}\norml{\s_{i+1} }^3 \right) \geqslant  \frac{1}{2} \sum_{i=1}^{k} \left(\norml{\s_{i} }^3 + \norml{\s_{i+1} }^3 \right).  \label{cov_5}
	\end{align}
	Plugging \cref{cov_5} into \cref{cov_6} yields that
	\begin{align*}
	f(\x_{k+1}) - f(\x_0) &\leqslant  - \sum_{i=1}^{k}    \left(\frac{3M-2L_2}{24} - \frac{3}{2}\beta -\frac{3}{4}\alpha \right)  \left(\norml{\s_{i} }^3 + \norml{\s_{i+1} }^3 \right) +    \left(\beta + \frac{\alpha}{2}\right) \norm{\s_{0}}^3 + \left(\beta + \frac{\alpha}{2}\right) \sum_{i=0}^{k}\epsilon_1^3\\
	&\numleqslant{i} - \sum_{i=1}^{k}  \left(\frac{3M-2L_2}{24} - \frac{5}{2}\beta -\frac{5}{4}\alpha \right)  \left(\norml{\s_{i} }^3 + \norml{\s_{i+1} }^3 \right) +    \left(\beta + \frac{\alpha}{2}\right) \norm{\s_{0}}^3 +\left(\beta + \frac{\alpha}{2}\right) \epsilon_1^3 , 
	\end{align*}
	where (i) follows from the fact that before the algorithm terminates we always have that $\norm{s_i} \geqslant \epsilon_1$ or $\norm{s_{i+1}} \geqslant \epsilon_1$, which gives that $\norm{s_i} ^3 + \norm{s_{i+1}} ^3 \geqslant \epsilon_1^3$.
	Therefore, we have
	\begin{align*}
	\sum_{i=1}^{k}  \left(\frac{3M-2L_2}{24} - \frac{5}{2}\beta -\frac{5}{4}\alpha \right)  \left(\norml{\s_{i} }^3 + \norml{\s_{i+1} }^3 \right) &\leqslant f(\x_0)  - f^*  +    \left(\beta + \frac{\alpha}{2}\right) \norm{\s_{0}}^3 +\left(\beta + \frac{\alpha}{2}\right) \epsilon_1^3 \\
	&\numequ{i}   f(\x_0)  - f^*    +\left(2\beta +  {\alpha} \right) \epsilon_1^3 \numberthis \label{cov_7}
	\end{align*}
	where (i) follows from the fact that $\norm{\s_{0}} = \epsilon_1$. Thus, if the algorithm never terminates, then we always have that   $\norm{s_i} \geqslant \epsilon_1$ or $\norm{s_{i+1}} \geqslant \epsilon_1$, which gives   $\norm{s_i} ^3 + \norm{s_{i+1}} ^3 \geqslant \epsilon_1^3$. Following from \Cref{cov_7}, we obtain that
	\begin{align}
	k \times \gamma \epsilon_1^3 \leqslant f(\x_0)  - f^*    +\left(2\beta +  {\alpha} \right) \epsilon_1^3,  \numberthis \label{cov_8}
	\end{align}
	where $\gamma \triangleq  \left(\frac{3M-2L_2}{24} - \frac{5}{2}\beta -\frac{5}{4}\alpha \right)$. Therefore, we obtain
	\begin{align}
	k \leqslant \frac{f(\x_0)  - f^*    +\left(2\beta +  {\alpha} \right) \epsilon_1^3}{ \gamma \epsilon_1^3 },
	\end{align}
	which shows that the algorithm must terminates if the total number of iterations exceeds $O(  {\epsilon_1^{-3}})$. With the choice of $\epsilon_1$ in \Cref{convergence_thm} , we obtain that the algorithm terminates at most with total iteration $k = O(\epsilon^{-3/2})$. 
	
	Suppose that the algorithm terminates at iteration $k$, then according to the analysis in \cref{cov_7}, we have that
	\begin{align*}
	\sum_{i=1}^{k-1}   \gamma \left(\norml{\s_{i} }^3 + \norml{\s_{i+1} }^3 \right) &\leqslant   f(\x_0)  - f^*    +\left(2\beta +  {\alpha} \right) \epsilon_1^3. \numberthis \label{cov_9}
	\end{align*}
	On the other hand, according to \cref{cov_9} and the terminal condition that  $\norm{s_i} \leqslant \epsilon_1$ and $\norm{s_{i+1}} \leqslant \epsilon_1$, we obtain  
	\begin{align*}
	\sum_{i=1}^{k}   \gamma \left(\norml{\s_{i} }^3 + \norml{\s_{i+1} }^3 \right) &\leqslant   f(\x_0)  - f^*    +\left(2\beta +  {\alpha} + 2\gamma \right) \epsilon_1^3,
	\end{align*}
	which gives that
	\begin{align}
	\sum_{i=1}^{k+1} \norml{\s_{i} }^3 \leqslant \frac{f(\x_0)  - f^*    +\left(2\beta +  {\alpha} + 2\gamma \right) \epsilon_1^3}{\gamma}.
	\end{align}
	
	We next consider the convergence of $\norml{\nabla f(x_k)}$ and $\norml{\nabla^2 f(x_k)}$.
	Next, we prove the convergence rate of $\nabla f(\cdot)$ and $\nabla^2 f(\cdot)$. We first derive 
	\begin{align*}
	\norml{\nabla f(\x_{m+1})} &\numequ{i} \normlarge{ \nabla f(\x_{m+1}) - \left(\mathbf{g}_m + \mathbf{H}_m \s_{m+1} + \frac{M}{2}\norml{\s_{m+1}} \s_{m+1} \right) } \\
	&\leqslant \normlarge{ \nabla f(\x_{m+1}) - \left(\mathbf{g}_m + \mathbf{H}_k \s_{m+1} \right) } + \frac{M}{2} \norml{\s_{m+1}}^2 \\
	&\leqslant \normlarge{ \nabla f(\x_{m+1}) - \nabla f(\x_m) -  \nabla^2 f(\x_m)\s_{m+1} } + \norml{\nabla f(\x_m) - \mathbf{g}_m } + \norml{(\nabla^2 f(\x_m) - \mathbf{H}_m )\s_{m+1}}  + \frac{M}{2} \norml{\s_{m+1}}^2 \\
	&\numleqslant{ii} \frac{L_2}{2} \norml{\s_{m+1}}^2 + \beta (\norml{\s_{m}}^2 + \epsilon_1^2) + \alpha (\norml{\s_{m}}+ \epsilon_1)\norml{\s_{m+1}}  + \frac{M}{2} \norml{\s_{m+1}}^2 \\
	&\numleqslant{iii}   \left(\frac{L+M}{2} + 2\beta + 2\alpha \right)\epsilon_1^2 \numleqslant{iv} \epsilon  ,
	\end{align*}
	where (i) follows from \cref{opt_1} with $\mathbf{g} = \mathbf{g}_m, \mathbf{H} = \mathbf{H}_m$ and $\s = \s_{m+1}$, (ii) follows from \cref{gradient_bound} in \Cref{Hessian_square_bound} and \Cref{assumption},   (iii) follows from the terminal condition of the algorithm, and (iv) follows from \cref{cov_10}.
	
	Similarly, we have 
	\begin{align*}
	\hspace{-8mm}\nabla^2 f(\x_{m+1}) &\overset{(i)}{\succcurlyeq} \mathbf{H}_m - \norml{\mathbf{H}_m -\nabla^2 f(\x_{m+1})} \mathbf{I}  \\
	&\overset{(ii)}{\succcurlyeq} - \frac{M}{2} \norml{\s_{m+1}} \mathbf{I}  - \norml{\mathbf{H}_m -\nabla^2 f(\x_{m+1})} \mathbf{I} \\
	&\succcurlyeq  - \frac{M}{2} \norml{\s_{m+1}} \mathbf{I}  - \norml{\mathbf{H}_m -\nabla^2 f(\x_{m })}\mathbf{I} - \norml{\nabla^2 f(\x_{m}) -\nabla^2 f(\x_{m+1})}\mathbf{I}  \\
	&\overset{(iii)}{\succcurlyeq}  - \frac{M}{2} \norml{\s_{m+1}} \mathbf{I}  -  \alpha (\norml{\s_{m}} + \epsilon_1)\mathbf{I}  -  L_2 \norml{\s_{m+1}}\mathbf{I} \\
	&\overset{(iv)}{\succcurlyeq} - \left( \frac{M+2L_2}{2}+ 2\alpha \right)\epsilon_1  \mathbf{I}  \overset{( v)}{\succcurlyeq} \epsilon \mathbf{I},
	\end{align*}
	where (i) follows from Weyl's inequality, (ii) follows from \cref{opt_2} with $\mathbf{H} = \mathbf{H}_{m}$ and $\s = \s_{m+1}$, (iii) follows from \Cref{assumption} and the fact that $\nabla^2 f(\cdot)$ is $L_2$-Lipschitz,   (iv) follows from  the terminal condition of the algorithm, and (v) follows from \cref{cov_10}.
\end{proof}

 \section{Proofs for SVRC under Sampling with Replacement}

\subsection{Proof of \Cref{With_Replacement_Gradient}}

The idea of the proof is to apply the following matrix Bernstein inequality \cite{Tropp2012} for sampling with replacement  to characterize the sample complexity in order to satisfy the inexactness condition $\norml{  \Hb_k - \nabla^2 F(\x_k) } \leqslant  \alpha \max \{\norml{\s_k}, \epsilon_1\} $ with the probability at least $1 - \zeta$. 

%-----------------------------------------------------------
% Matrix Bernstein Inequality 
%----------------------------------------------------------- 
\begin{lemma}[Matrix Bernstein Inequality] \label{bernstern_0}
	Consider a finite sequence $\{\X_k\}$ of independent, random matrices with dimensions $d_1 \times d_2$. Assume that each random matrix satisfies 
	\begin{align*}
	\Ebb \X_k = \mathbf{0} \quad \text{and} \quad \norml{\X_k} \leqslant R  \quad \text{almost surely}.
	\end{align*}
	Define  
	\begin{align}
	\sigma^2 \triangleq \max \left(	\norml{\sum\nolimits_{k} \Ebb(\X_k\X_k^*) }, \norml{\sum\nolimits_{k} \Ebb(\X_k^*\X_k) } \right).
	\end{align}
	Then, for all $\epsilon \geqslant 0$,
	\begin{align*}
	P   \bigg(\norml{   \sum\nolimits_k\X_{k}  }  \geqslant   \epsilon \bigg)  \leqslant     2(d_1 + d_2) \exp \bigg( -  \frac{  \epsilon^2/2 }{  \sigma^2  + R\epsilon/3}\bigg).
	\end{align*}  
\end{lemma}

Let   $\xi_{H}(k) $ be the collection of index that uniformly picked from $1, \cdots, N$ with replacement, and  $\X_i$ be
\begin{align*}
\X_i = \frac{1}{|\xi_{H}(k) |} \left(\nabla^2 f_{ i}(\x_k) -  \nabla^2 f_{ i}(\tilde{\x})  +  \nabla^2 F(\tilde{\x}) - \nabla^2 F( \x_{k})\right),
\end{align*}
 then we have
\begin{align} \label{with_replacement_Hessian_sample_size}
\Hb_k - \nabla^2 F(\x_k) = \sum_{i \in    \xi_{H}(k)  } \X_i.
\end{align}

 Moreover, we have $\mathbb{E} \X_i = \mathbf{0}$, and
\begin{align*}
R \triangleq \norml{\X_i} &= \frac{1}{|\xi_{H}(k) |} \norml{\nabla^2 f_{\xi_i}(\x_k) -  \nabla^2 f_{\xi_i}(\tilde{\x})  +  \nabla^2 F(\tilde{\x}) - \nabla^2 F( \x_{k})} \\
&\numleqslant{i}\frac{ 2L_2 }{|\xi_{H}(k) |} \norml{\x_{k} - \tilde{\x}}, \numberthis \label{bernstern_1}
\end{align*}
where (i) follows because $\nabla^2 f_i(\cdot)$ is $L_2$ Lipschitz, for $1 \leqslant i \leqslant N$.

The variance also can be bounded by
  \begin{align*}\label{bernstern_2}
\sigma^2 &\triangleq \max \left(	\norml{\sum\nolimits_{k \in \xi_{H}(k)} \Ebb(\X_k\X_k^*) }, \norml{\sum\nolimits_{k \in \xi_{H}(k)} \Ebb(\X_k^*\X_k) } \right) \\
&\numleqslant{i} 	\norml{\sum\nolimits_{k \in \xi_{H}(k)} \Ebb(\X_k^2 ) } \numleqslant{ii}   \sum\nolimits_{k \in \xi_{H}(k)} \Ebb \norml{\X_k^2}  \leqslant \sum\nolimits_{k \in \xi_{H}(k)}\Ebb \norml{\X_k}^2  \\
& \numleqslant{ii} \frac{ 4L_2^2 }{|\xi_{H}(k) |} \norml{\x_{k} - \tilde{\x}}^2 \numberthis
\end{align*}
where (i) follows from the fact that $\X_k$ is real and symmetric, (ii) follows from   Jensen's inequality, and (iii) follows from \cref{bernstern_1}.

Therefore, in order to satisfy $\norml{   \Hb_k - \nabla^2 F(\x_k)   } \leqslant  \alpha \max \{\norml{\s_k}, \epsilon_1\}$ with probability at least $1 - \zeta$, by \cref{with_replacement_Hessian_sample_size}, it is equivalent to require $\norml{ \sum_{i \in  \xi_{H}(k)} \X_i    } \leqslant  \alpha \max \{\norml{\s_k}, \epsilon_1\}$ with probability at least $1 - \zeta$.  We now apply \Cref{bernstern_0} for $\X_i$,  and it is  sufficient to have:
 \begin{align*}
  		  2(d_1 + d_2) \exp \bigg(   \frac{-  \epsilon^2/2 }{  \sigma^2  + R\epsilon/3 }\bigg) \leqslant \zeta \\
\end{align*}  
which is equivalent to have
   	\begin{align*}
	   	\frac{  1 }{   \sigma^2   + R\epsilon/3 } \geqslant \frac{2}{\epsilon^2} \log \left(\frac{2(d_1 + d_2)}{\zeta}\right). \numberthis \label{with_1}
    \end{align*}
Plugging \cref{bernstern_1,bernstern_2} into \cref{with_1} yields 
   	\begin{align*}
   	 	\frac{  1 }{   \frac{ 4L_2^2 }{|\xi_{H}(k) |} \norml{\x_{k} - \tilde{\x}}^2   + \frac{ 2L_2 }{|\xi_{H}(k) |} \norml{\x_{k} - \tilde{\x}} \epsilon/3 }   \geqslant \frac{2}{\epsilon^2}  \log \left(\frac{ 4d}{\zeta}\right) , 
   	\end{align*}
which gives
    	\begin{align}
    	|\xi_{H}(k) | \geqslant   \left(    \frac{ 8L_2^2 }{\epsilon^2} \norml{\x_{k} - \tilde{\x}}^2   + \frac{ 4L_2 }{3 \epsilon} \norml{\x_{k} - \tilde{\x}}    \right) \log \left(\frac{ 4d}{\zeta}\right) .
    	\end{align}
Substituting $ \epsilon =   \alpha \max \{\norml{\s_k}, \epsilon_1\}$, we obtain the required sample size to be bounded by
\begin{align}
|\xi_{H}(k) | \geqslant   \left(    \frac{ 8L_2^2 }{ \alpha^2 \max \{\norml{\s_k}^2, \epsilon_1^2\}} \norml{\x_{k} - \tilde{\x}}^2   + \frac{ 4L_2 }{3  \alpha \max \{\norml{\s_k}, \epsilon_1\}} \norml{\x_{k} - \tilde{\x}}    \right) \log \left(\frac{ 4d}{\zeta}\right) .
\end{align}
%Combining with \cref{cov_10} which gives $\epsilon_1 = \tau \sqrt{\epsilon}$, we obtain 
%\begin{align}
%|\xi_{H}(k) | \geqslant   \left(    \frac{ 8L_2^2 }{ \alpha^2 \max \{\norml{\s_k}^2, \epsilon_1^2\}} \norml{\x_{k} - \tilde{\x}}^2   + \frac{ 4L_2 }{3  \alpha \max \{\norml{\s_k}, \epsilon_1\}}  \norml{\x_{k} - \tilde{\x}}    \right) \log \left(\frac{ 4d}{\zeta}\right) .
%\end{align}

We next bound  the sample size $|\xi_{g}(k)|$ for the gradient in the similar procedure.  We first define $\X_i \in \mathbb{R}^{d  \times 1}$ as
\begin{align}  
\X_i = \frac{1}{|\xi_{g}(k)|} \left(\nabla f_{\xi_i}(\x_k) -  \nabla f_{\xi_i}(\tilde{\x})  +  \nabla F(\tilde{\x}) - \nabla F( \x_{k})\right),
\end{align}
then we have
\begin{align}\label{with_replacement_gradient_norm_bound}
	\g_k - \nabla f(\x_k) = \sum_{i \in  \xi_{g}(k) } \X_i
\end{align}
Furthermore,
\begin{align*} 
R = \norml{\X_i} &=  \frac{1}{|\xi_{g}(k)|} \norml{\nabla f_{\xi_i}(\x_k) -  \nabla f_{\xi_i}(\tilde{\x})  +  \nabla F(\tilde{\x}) - \nabla F( \x_{k})}\numleqslant{i}\frac{2L_1}{|S_{g,k}|} \norml{\x_{k} - \tilde{\x}}, \numberthis
\end{align*}
where (i) follows because $\nabla  f_i(\cdot)$ is $L_1$ Lipschitz, for $i=1,\ldots,N$, and
\begin{align*}
\sigma^2 &\triangleq \max \left(	\norml{\sum\nolimits_{k \in \xi_{g}(k)} \Ebb(\X_k\X_k^*) }, \norml{\sum\nolimits_{k \in \xi_{g}(k)} \Ebb(\X_k^*\X_k) } \right)   \leqslant \sum\nolimits_{k \in \xi_{H}(k)}\Ebb \norml{\X_k}^2  \\
& \numleqslant{ii} \frac{ 4L_1^2 }{|\xi_{g}(k) |} \norml{\x_{k} - \tilde{\x}}^2 
\end{align*}

In order to satisfy $\norml{   \g_k - \nabla  F(\x_k)   } \leqslant \beta \max \left\{ \norml{  \s_{k}}^2 ,\epsilon_1^2  \right\}  $ with the probability at least $1 - \zeta$, by \cref{with_replacement_gradient_norm_bound}, it is equivalent to require $\norml{    \sum_{i \in  \xi_{g}(k)| }\X_i    } \leqslant  \beta \max \left\{ \norml{  \s_{k}}^2 ,\epsilon_1^2  \right\}$ with the probability at least $1 - \zeta$. We then apply \Cref{bernstern_0} for $\X_i$ in the way similar to that for bounding the sample size for Hessian, with $R = \frac{2L_1}{|S_{g,k}|} \norml{\x_{k} - \tilde{\x}}$,  $\epsilon =\beta \max \left\{ \norml{  \s_{k}}^2 ,\epsilon_1^2  \right\}  $, and $\sigma^2 =\frac{ 4L_1^2 }{|\xi_{g}(k) |} \norml{\x_{k} - \tilde{\x}}^2$, and obtain the required sample size to satisfy
\begin{align}
|\xi_{g}(k)| \geqslant \left(    \frac{ 8L_1^2 }{ \beta^2 \max \{\norml{\s_k}^4, \epsilon_1^4\}} \norml{\x_{k} - \tilde{\x}}^2   + \frac{ 4L_1 }{3  \beta \max \{\norml{\s_k}^2, \epsilon_1^2\}} \norml{\x_{k} - \tilde{\x}}    \right) \log \left(\frac{ 2(d+1)}{\zeta}\right) .
\end{align}

\subsection{Proof of \Cref{Total_Hessian_complexity_bound}}

First, by \cref{vector_norm_bound_7}, we have 
	\begin{align} \label{cubic_bound}
	\sum_{i=1}^{k+1} \norml{\x_{i} - \x_{i-1}}^3 \leqslant C.
	\end{align}

We then  derive
	\begin{align*} \label{bound_for_sum_1}
	\sum_{i=0}^{k/m-1}& \sum_{j=1}^{m-1} \norml{\x_{i\cdot m + j } - \x_{i\cdot m}}^2 \\
	& \leqslant 	\sum_{i=0}^{k/m-1} \sum_{j=1}^{m-1} \bigg(   \norml{\x_{i\cdot m  + j} - \x_{i\cdot m+j-1 }}+\cdots + \norml{\x_{i\cdot m  + 1} - \x_{i\cdot m }}\bigg)^2\\
	&\leqslant  	\sum_{i=0}^{k/m-1} \sum_{j=1}^{m-1} \bigg(   \norml{\x_{i\cdot m  + m -1} - \x_{i\cdot m+m-2 }}+\cdots + \norml{\x_{i\cdot m  + 1} - \x_{i\cdot m }}\bigg)^2\\
	&= \sum_{i=0}^{k/m-1} \sum_{j=1}^{m-1} \bigg( \sum_{l = 1}^{m-1} \norml{\x_{i\cdot m + l } - \x_{i\cdot m +l-1}}\bigg)^2 
	\numleqslant{i} \sum_{i=0}^{k/m-1} \sum_{j=1}^{m-1}   m  \sum_{l = 1}^{m-1} \norml{\x_{i\cdot m + l } - \x_{i\cdot m +l-1}}^2\\
	& \numleqslant{ii} m^{2} \sum_{i=0}^{k/m-1}       \sum_{l = 1}^{m-1} \norml{\x_{i\cdot m + l } - \x_{i\cdot m +l-1}}^2    \leqslant m^2 \sum_{i=1}^{k}  \norml{\x_{i} - \x_{i-1}}^2\\
	&\numleqslant{iii} m^2 k^{1/3} \big(\sum_{i=1}^{k}  \norml{\x_{i} - \x_{i-1}}^3\big)^{2/3} 
	\numleqslant{iv} m^2 k^{1/3} \CubicBound^{2/3},	\numberthis
	\end{align*}
where (i) follows from the Cauthy-Schwaz inequality 
%$\big( \sum_{i = 1}^{m} a_i \big)^2 \leqslant m^2 \sum_{i = 1}^{m} a_i^2 $, 
(ii) follows because $j$ is not a variable in the inner summation, (iii) follows from Holder's inequality, and (iv) follows from \cref{cubic_bound}.

Similarly, we have that
\begin{align*}
	\sum_{i=0}^{k/m-1}  \sum_{j=1}^{m-1} \norml{\x_{i\cdot m + j } - \x_{i\cdot m}} 	& \leqslant 	\sum_{i=0}^{k/m-1} \sum_{j=1}^{m-1} \bigg(   \norml{\x_{i\cdot m  + j} - \x_{i\cdot m+j-1 }}+\cdots + \norml{\x_{i\cdot m  + 1} - \x_{i\cdot m }}\bigg) \\
	&\leqslant  	\sum_{i=0}^{k/m-1} \sum_{j=1}^{m-1} \bigg(   \norml{\x_{i\cdot m  + m -1} - \x_{i\cdot m+m-2 }}+\cdots + \norml{\x_{i\cdot m  + 1} - \x_{i\cdot m }}\bigg) \\
	&= \sum_{i=0}^{k/m-1} \sum_{j=1}^{m-1} \bigg( \sum_{l = 1}^{m-1} \norml{\x_{i\cdot m + l } - \x_{i\cdot m +l-1}}\bigg)  
	\numleqslant{i} m \sum_{i=0}^{k/m-1}     \sum_{l = 1}^{m-1} \norml{\x_{i\cdot m + l } - \x_{i\cdot m +l-1}} \\
	&\leqslant m \sum_{i=1}^{k}  \norml{\x_{i} - \x_{i-1}} \numleqslant{ii} m  k^{2/3} \bigg(\sum_{i=1}^{k}  \norml{\x_{i} - \x_{i-1}}^3\bigg)^{1/3} 
	\numleqslant{iii} m  k^{2/3} \CubicBound^{1/3},	\numberthis \label{with_first_order}
\end{align*}
where (i) follows because $j$ is not a variable in the inner summation, (ii) follows from Holder's inequality, and (iii) follows from \cref{cubic_bound}.
	
Thus, the total sample size for Hessian is given by
	\begin{align*}
	m + \frac{kN}{m}& +  \sum_{i=0}^{k/m-1}  \sum_{j=1}^{m-1} |\xi_{H}(k) | \\
	 &\numleqslant{i}   \frac{CkN}{m} +   \sum_{i=0}^{k/m-1} \sum_{j=1}^{m-1}  \left(    \frac{ 8L_2^2 }{ \alpha^2 \max \{\norml{\s_k}^2, \epsilon_1^2\}}  \norml{\x_{i\cdot m + j } - \x_{i\cdot m}}^2    + \frac{ 4L_2 }{3  \alpha \max \{\norml{\s_k}, \epsilon_1\}}  \norml{\x_{i\cdot m + j } - \x_{i\cdot m}}      \right) \log \left(\frac{ 4d}{\zeta}\right)  \\
	& \leqslant  	\frac{CkN}{m} +   \sum_{i=0}^{k/m-1} \sum_{j=1}^{m-1}  \left(    \frac{ 8L_2^2 }{ \alpha^2  \epsilon_1^2 }  \norml{\x_{i\cdot m + j } - \x_{i\cdot m}}^2   + \frac{ 4L_2 }{3  \alpha  \epsilon_1 }  \norml{\x_{i\cdot m + j } - \x_{i\cdot m}}    \right) \log \left(\frac{ 4d}{\zeta}\right)  \\
	&\numleqslant{ii} \frac{CkN}{m} +     \left( \frac{ 8L_2^2 }{ \alpha^2  \epsilon_1^2 }  m^2 k^{1/3} \CubicBound^{2/3}   + \frac{ 4L_2 }{3  \alpha  \epsilon_1 } m  k^{2/3} \CubicBound^{1/3}  \right)\log \left(\frac{ 4d}{\zeta}\right)   \\
	&\numleqslant{iii} \log \left(\frac{ 4d}{\zeta}\right)  \left( \frac{ N}{m\epsilon^{3/2}} +      \frac{  C }{    \epsilon^{3/2} }  m^2      + \frac{C }{  \epsilon^{3/2} } m   \right) =  \log \left(\frac{ 4d}{\zeta}\right)   \frac{C}{\epsilon^{3/2}} \left(  \frac{ N}{m } +   m^2     \right)
	\end{align*}
where (i) follows form  \Cref{With_Replacement_Gradient}, and (ii) follows form \cref{bound_for_sum_1,with_first_order},
(iii) follows from the fact that  $\zeta \leqslant 1$ and $d \geqslant 1$ which gives $\log \left(\frac{ 4d}{\zeta}\right) > 1$,  and  $\epsilon_1 = O(\epsilon^{1/2})$ such that $ k = O(\epsilon^{-3/2})$ according to \Cref{convergence_thm}

We minimize the above bound over $m$, substitute the minimizer $m^\star = N^{1/3}$, and obtain
	\begin{align*}
   \sum_{i=0}^{k} |\xi_{H}(k)|\leqslant   \frac{ C N^{2/3} }{\epsilon^{ 3/2}}\log \left(\frac{ 4d}{\zeta}\right)  .
 \end{align*}

 Next, according to \Cref{With_Replacement_Gradient}, \Cref{assumption} is satisfies with probability at least $1 - \zeta$ for gradient and $1 - \zeta$ for Hessian . Thus, according to the union bound, the probability of a failure satisfaction per iteration is at most $2 \zeta$. Then, for $k$ iteration, the probability of failure satisfaction of \Cref{assumption} is at most $2 k \zeta$ according to the union bound. To obtain  \Cref{assumption} holds for the total $k$ iteration with probability least $1 - \delta$, we require
\begin{align*}
 1 - 2k\zeta \geqslant 1 - \delta,
\end{align*}
which yields
\begin{align*}
\zeta \leqslant  \frac{\delta}{2k}.
\end{align*} 
Thus, with probability $1 - \delta$, the algorithms successfully outputs an $\epsilon$ approximated second-order stationary point, with the total Hessian sample complexity is bounded by
\begin{align}
	 \sum_{i=0}^{k} |\xi_{H}(k)|\leqslant   \frac{ C N^{2/3} }{\epsilon^{ 3/2}}\log \left(\frac{ 8d}{\epsilon^{3/2} \delta  }\right)  \leqslant   \frac{ C N^{2/3} }{\epsilon^{ 3/2}}\log \left(\frac{ 8d}{\epsilon  \delta  }\right) .
\end{align} which gives
\begin{align}
	 \sum_{i=0}^{k} |\xi_{H}(k)| = \tilde{O} \left(\frac{  N^{2/3} }{\epsilon^{ 3/2}}\right).
\end{align}

%%%

\section{Proof of Concentration Inequality for Sampling without replacement }
The proof generalizes the Hoeffding-Serfling inequality for scalar random variables in \cite{bardenet2015} to that for random matrices. We also apply various properties for handling random matrices in \cite{Tropp2012}.

\subsection{Definitions and Useful Lemmas}

We first introduce the definition of the matrix function following  \cite{Tropp2012}, and then introduce a number of Lemmas that are useful in the proof.

Given a symmetric matrix $\A$, suppose its eigenvalue decomposition is given by  $ \A = \U \bLambda \U^T \in \mathbb{R}^{d \times d}$, where $\bLambda  = diag(\lambda_1, \cdots, \lambda_d)$. Then a function $f: \mathbb{R}  \rightarrow \mathbb{R} $ of $\A$ is defined as:
\begin{align}
	f(\A) \triangleq \U f(\bLambda ) \U^T ,
\end{align} 
 where $f(\bLambda ) = diag(f(\lambda_1), \cdots, f(\lambda_d))$, i.e., $f(\bLambda )$ applies the function $f(\cdot)$ to each diagonal entry of the matrix $\bLambda $. 
  
The trace exponential function $\mathrm{tr} \ \exp : \A \rightarrow \mathrm{tre}^\A$, i.e., $\tre (\A)$, is defined to first apply the exponential matrix function $\exp(\A)$, and then take the trace of $\exp(\A)$. Such a function is monotone with respect to the semidefinite order:
 \begin{align} \label{trace_exp_monotone}
 	\A \preccurlyeq  \mathbf{H}  \quad \Longrightarrow \quad \tre(\A) \preccurlyeq \tre \mathbf{(H)},
 \end{align}
 which follows because for two symmetric matrices $\A$ and $\mathbf{ H}$, if $\A \preccurlyeq \mathbf{H}$, then $\lambda_i(\A) \leqslant \lambda_i \mathbf{ (H) }$ for every $i$, where $\lambda_i(\A)$ is the $i$-th largest eigenvalue of $\A$. Furthermore, the matrix function  $\log(\cdot)$ is monotone with respect to the semidefinite order (see the exercise 4.2.5 in \cite{Bhatia07}):
 \begin{align} \label{log_monotone}
 	 \mathbf{ 0} \prec \A \preccurlyeq \mathbf{ H} \quad  \Longrightarrow   \quad \log(\A) \preccurlyeq \log(\mathbf{ H}).
 \end{align}
 
The next three lemmas follow directly from \cite{bardenet2015} because the proofs are applicable for matrices. 
\begin{lemma} \label{HS_lemma_1_1}[\cite{bardenet2015}]
Let $\Z_k \triangleq \frac{1}{k} \sum_{i = 1}^{k} \X_i$. The following reverse martingale structure holds for $\{\Z_k\}_{k \leqslant N}$:
	\begin{align}
		\mathbb{E} [\Z_k| \Z_{k+1}, \cdots \Z_{N-1}] = \Z_{k+1}.
	\end{align}
\end{lemma}

\begin{lemma}\label{HS_lemma_1_2}[\cite{bardenet2015}]
Let $\Y_k \triangleq \Z_{N-k}$  for $1 \leqslant k \leqslant N-1$. For any $\lambda > 0 $, the following equality holds for $2 \leqslant k \leqslant n$,
	\begin{align}
		\lambda \Y_k = \lambda \Y_{k-1} - \lambda \frac{\X_{N-k+1} - \mu - \Y_{k-1}}{N-k}.
	\end{align}
\end{lemma}

\begin{lemma}\label{HS_lemma_1_3}[\cite{bardenet2015}]
Let $\Y_k \triangleq \Z_{N-k}$ for $1 \leqslant k \leqslant N-1$. For $2 \leqslant k \leqslant N$, the following equality holds
	\begin{align}
		\mathbb{E} [\X_{N-k+1} - \mu - \Y_{k-1} | Y_1, \cdots, \Y_{k-1}] = 0,
	\end{align}
	where $\mu = \frac{1}{N} \sum_{t=1}^{N} \X_t$.
\end{lemma}

The following lemma is an extension of Hoeffding's inequality for scalars to matrices. We include a brief proof for completeness.
\begin{lemma}[\textbf{Hoeffding's Inequality for Matrix}] \label{HS_lemma_1_4}
For a random symmetric matrix $\X \in \RR^{d \times d}$, suppose 
  \begin{align*}
  	\mathbb{E} [\X] = 0 \quad \text{   and   }  \quad a\I \preccurlyeq \X \preccurlyeq b \I.
  \end{align*}
where $a$ and $b$ are real constants. Then for any $\lambda > 0$, the following inequality holds
  \begin{align}
   \mathbb{E} [e^  {\lambda \X} ] \preccurlyeq \exp \bigg(\frac{1}{8} \lambda^2 (b - a) ^2 \I  \bigg) .
  \end{align}
\end{lemma}
\begin{proof}
The proof follows from the standard reasoning for scalar version. We emphasize only the difference in handling matrices. Suppose the eigenvalue decomposition of the symmetric random matrix  $\X$ can be written as $\X = \U \Lambda \U^T$, where $\U = [\mathbf{u}_1, \cdots, \mathbf{u}_d]$ and $\mathbf{\Lambda} = diag(\lambda_1, \cdots, \lambda_d)$. Therefore, we obtain $e^  {\lambda \X} = \sum_{i = 1}^{ d } e^{\lambda \lambda_i} \ub_i\ub_i^T$. 
	
	Since scalar function $e^{\lambda x}$ is convex for any $\lambda > 0$, for $ 1 \leqslant i \leqslant d$, we have  
	\begin{align}
	   e^{\lambda \lambda_i} \leqslant \bigg( \frac{ b - \lambda_i }{b -a} e^{\lambda a} + \frac{   \lambda_i - a }{b -a} e^{\lambda b} \bigg),
	\end{align} which  implies that
	\begin{align} \label{exp_convex}
		e^{\lambda \lambda_i} \ub_i \ub_i^T\preccurlyeq \bigg( \frac{ b - \lambda_i }{b -a} e^{\lambda a} + \frac{   \lambda_i - a }{b -a} e^{\lambda b} \bigg) \ub_i \ub_i^T.
	\end{align}

Then,
\begin{align*}
	\mathbb{E} [e^{\lambda \X}] &=\mathbb{E}  \bigg[\sum_{i = 1}^{ d } e^{\lambda \lambda_i} \ub_i\ub_i^T\bigg] 
	\overset{\text{(i)}} \preccurlyeq \mathbb{E}\bigg[\sum_{i = 1}^{ d }\bigg( \frac{ b - \lambda_i }{b -a} e^{\lambda a} + \frac{   \lambda_i - a }{b -a} e^{\lambda b} \bigg) \ub_i \ub_i^T\bigg] \\
	& = \mathbb{E}\bigg[	\sum_{i = 1}^{ d } \frac{ b }{b -a} e^{\lambda a} \ub_i \ub_i^T
	- \sum_{i = 1}^{ d } \frac{ \lambda_i }{b -a} e^{\lambda a} \ub_i \ub_i^T + \sum_{i = 1}^{ d } \frac{ \lambda_i }{b -a} e^{\lambda b} \ub_i \ub_i^T - \sum_{i = 1}^{ d } \frac{ a }{b -a} e^{\lambda b} \ub_i \ub_i^T \bigg]\\
	& \numequ{ii} \mathbb{E}\bigg[ 	\sum_{i = 1}^{ d } \frac{ b }{b -a} e^{\lambda a} \ub_i \ub_i^T -     \frac{ e^{\lambda a} }{b -a} \X  +  \frac{ e^{\lambda b} }{b -a} \X  - \sum_{i = 1}^{ d } \frac{ a }{b -a} e^{\lambda b} \ub_i \ub_i^T \bigg] \\
	& \numequ{iii} \mathbb{E}\bigg[ 	\sum_{i = 1}^{ d } \frac{ b }{b -a} e^{\lambda a} \ub_i \ub_i^T  - \sum_{i = 1}^{ d } \frac{ a }{b -a} e^{\lambda b} \ub_i \ub_i^T  \bigg] \\
	& \numequ{iv} \mathbb{E}\bigg[  \frac{ b }{b -a} e^{\lambda a} \I   -  \frac{ a }{b -a} e^{\lambda b} \I \bigg] 
	=   \bigg(\frac{ b }{b -a} e^{\lambda a}   -  \frac{ a }{b -a} e^{\lambda b}\bigg)    \I  \\
	 & \preccurlyeq  \exp \bigg(\frac{1}{8} \lambda^2 (b-a)^2 \bigg) \I 
	 \numequ{v}\exp \bigg(\frac{1}{8} \lambda^2 (b-a)^2 \I\bigg),
	\numberthis
\end{align*}
where (i) follows from \cref{exp_convex} and the fact that  the expectation of random matrix preserves the semi-definite order,  (ii) follows from $\X = \sum_{i=1}^{d} \lambda_i \ub_i \ub_i^T$, (iii) follows because $\mathbb{E} [\X] = 0$, (iv) follows because $\I = \U\U^T = \sum_{i = 1}^{d} \ub_i\ub_i^T $, and (v) follows from the standard steps in the proof of the scalar version of Hoeffding's inequality.
\end{proof}
  
\begin{lemma} \label{HS_lemma_1_5}\cite{Tropp2012}[Corollary 3.3]
	 Let $\mathbf{H}$ be a fixed self-adjoint matrix, and let $\X$ be a random self-adjoint matrix. The following inequality holds
	 \begin{align}
	 	\mathbb{E} \ \mathrm{tr}  \exp (\mathbf{H} + \X) \leqslant \mathrm{tr} \exp (\mathbf{H} + \log (\mathbb{E}e^\X)).
	 \end{align}
\end{lemma}

\begin{lemma} \label{HS_square_sum_inequality}\cite{bardenet2015}
	For integer $n \leqslant N$, the following inequality holds
	\begin{align*}
		  \sum_{t=1}^{n}  \big(\frac{ 1 }{N-t}  \big)^2 \leqslant \frac{n}{(N-n)^2} \big(1 - \frac{n-1}{N}\big) 
	\end{align*}
\end{lemma}

\subsection{Proof of  \Cref{Matrix_Hoeffding_Serfling_Inequality}}

First, it suffices to show the theorem only for symmetric matrices, due to the technique of {\em dilations} in \cite{Tropp2012} that transforms the asymmetric matrix to a symmetric matrix while keeping the spectral norm to be the same.

Second, it also suffices to show that for $1 \leqslant i \leqslant N$, $\X_i$ are symmetric and bounded, i.e., $a\I \preccurlyeq  \X_i \preccurlyeq b\I$, and $1 \leqslant n \leqslant N-1$, the following inequality holds
		\begin{align*}
		P\bigg(\lambda_{\max}\bigg(\frac{1}{n}\sum_{i=1}^{n} \X_i - \mu\bigg) \geqslant \epsilon \bigg) \leqslant d \exp \bigg( -  \frac{n \epsilon^2}{2(b-a)^2  (1+1/n) (1- n/N)}\bigg) . 
		\end{align*}
This is because the above result, with  $\X_i$ being replaced with $- \X_i$, implies
\begin{align}\label{HS_symmetric_smallest}
	P\bigg( \lambda_{\min}\bigg(\frac{1}{n}\sum_{i=1}^{n}   \X_i - \mu\bigg) \leqslant - \epsilon \bigg) \leqslant d \exp \bigg( -  \frac{n \epsilon^2}{2(b-a)^2  (1+1/n) (1- n/N)}\bigg).
	\end{align}
Then the combination of the two results completes the desired theorem.

We start the proof by applying the matrix version of Chernoff inequality as follows. Let $\Z_k \triangleq \frac{1}{k} \sum_{i = 1}^{k} \X_i$, for any $\lambda >0$, we obtain
	\begin{align*} \label{HS_almost_done}
			P\bigg(\lambda_{\max}( \Z_n  ) \geqslant \epsilon \bigg) & = P\bigg(\exp (\lambda  \lambda_{\max}( \Z_n )) \geqslant \exp(\lambda \epsilon)
			 \bigg) \\
			 & \numleqslant{i} \exp(- \lambda \epsilon) \Ebb \exp\big(\lambda  \lambda_{\max}( \Z_n )\big) \\
			 & \numleqslant{ii} \exp(- \lambda \epsilon) \Ebb \ \lambda_{\max} \big(\exp(\lambda  \Z_n )\big) \\
			 &\numleqslant{iii} \exp(- \lambda \epsilon) \Ebb \  \tre (\lambda   \Z_n ) \\
			 & \numleqslant{iv} \exp(- \lambda \epsilon) \ \tre \bigg( \frac{\lambda^2}{2} (b-a)^2 \frac{(n+1)}{n^2}\bigg( 1 - \frac{n}{N} \bigg) I  \bigg) \\
			 & \numleqslant{v} d \exp \bigg(\frac{\lambda^2}{2} (b-a)^2 \frac{(n+1)}{n^2} \bigg( 1 - \frac{n}{N} \bigg)\bigg) \exp(- \lambda \epsilon) \\
			 & = d \exp \bigg(\frac{\lambda^2}{2} (b-a)^2 \frac{(n+1)}{n^2} \bigg( 1 - \frac{n}{N} \bigg) - \lambda \epsilon \bigg) \numberthis
	\end{align*}
where (i) follows from the matrix version of Chernoff inequality, (ii) follows from  the fact that  $\exp(\cdot)$ is an increasing function, thus $\exp\big(\lambda  \lambda_{\max}( \Z_n ) \big)= \lambda_{\max} \big(\exp(\lambda  \Z_n )$, and (iii) follows  from the fact that $\lambda_{\max}(\A)\leqslant \mathrm{ tr}(\A)$, with  $\A = \exp(\lambda \Z_n)$, we get the desire result.
%(iv) follows form \cref{HS_bound_tre} and  (v) follows form the equation $\mathrm{ tr} (aI) =  da$ if $I \in \mathbb{R}^{d\times d}$.

%%%%%%
We next bound $\mathbb{E} \  \mathrm{tr} \ \exp (\lambda \Z_{n})$. Let $Y_k \triangleq Z_{N-k}$ for $1 \leqslant k \leqslant N-1$, and $\mathbb{E}_k [\ \cdot \ ] \triangleq \mathbb{E} [\ \cdot \ | \Y_1, \cdots, \Y_k]$. Thus, 
	\begin{align*}  \label{iteration_main_function}
	  \mathbb{E} \  \mathrm{tr} \ \exp (\lambda \Y_n) &\numequ{i} \mathbb{E} \  \mathrm{tr} \ \exp \bigg( \lambda \Y_{n-1} -  \lambda \frac{\X_{N-n+1} - \mu - \Y_{n-1}}{N-n} \bigg) \\
	  & \numequ{ii} \mathbb{E} \  \Ebb_{n-1} \  \mathrm{tr}  \ \exp \bigg( \lambda \Y_{n-1} -  \lambda \frac{\X_{N-n+1} - \mu - \Y_{n-1}}{N-n} \bigg) \\
	  & \numleqslant{iii}  \mathbb{E}   \  \mathrm{tr} \   \exp \bigg( \lambda \Y_{n-1} + \log \Ebb_{n-1} \exp \bigg( -  \lambda \frac{\X_{N-n+1} - \mu - \Y_{n-1}}{N-n} \bigg) \bigg), \numberthis
	\end{align*}
where (i) follows from \Cref{HS_lemma_1_2}, (ii) follows from the tower property of expectation, (iii) follows by applying \Cref{HS_lemma_1_5}, where $ \lambda \Y_{n-1}$ is deterministic given ${\Y_1, \cdots, \Y_k}$, and $ -  \lambda  \big(\X_{N-n+1} - \mu - \Y_{n-1} \big)/{(N-n)}$ is a random variable matrix.

In order to apply \Cref{HS_lemma_1_4} to bound  $ \Ebb_{n-1} \exp ( -  \lambda  \big(\X_{N-n+1} - \mu - \Y_{n-1} \big)/{(N-n)})$, we first bound $ \X_{N-n+1} - \mu - \Y_{n-1}$ as follows:
\begin{align*} \label{HS_hoffeding_lemma_variance_bound_0}
	\X_{N-n+1} - \mu - \Y_{n-1} &\numequ{i} \X_{N-n+1} - \mu - \Z_{N - n+1} \\
	&\numequ{ii} \X_{N-n+1} - \mu -\frac{1}{N-n+1}\sum_{i = 1}^{N-n+1}\bigg( \X_i - \mu \bigg)\\
	&=  \X_{N-n+1}  -\frac{1}{N-n+1}\sum_{i = 1}^{N-n+1}  \X_i,  \numberthis
\end{align*}
where (i) follows from the definition of $\Y_{n-1}$ and (ii) follows from the definition of $\Z_{N - n +1}$. Since $a\I \preccurlyeq  \X_i \preccurlyeq b\I$, the above equality implies 
 
\begin{align} \label{HS_hoffeding_lemma_variance_bound_3}
-\frac{(b-a)}{N-n} \I	\preccurlyeq \frac{	\X_{N-n+1} - \mu - \Y_{n-1}}{N - n}  \preccurlyeq \frac{(b-a)}{N-n} \I.
\end{align} 
By applying \Cref{HS_lemma_1_4}, and the fact $\mathbb{E}_{n-1} [\X_{N-n+1} - \mu - \Y_{n-1}] = 0$ due to  \Cref{HS_lemma_1_3}, we obtain
\begin{align*} \label{iteration_constant}
	\mathbb{E}_{n-1} \exp \bigg( \X_{N-n+1} - \mu - \Y_{n-1} \bigg) &\preccurlyeq \exp \bigg(\frac{1}{8} \lambda^2 \bigg(\frac{2(b-a)}{N-n} \bigg) ^2 \I  \bigg)  
	= \exp \bigg(\frac{1}{2} \lambda^2 \bigg(\frac{  b-a  }{N-n} \bigg) ^2 \I  \bigg),  \numberthis
\end{align*}

Substituting \cref{iteration_constant} into \cref{iteration_main_function}, we obtain
\begin{align*} \label{HS_iteration}
	 \mathbb{E} \  \mathrm{tr} \ \exp (\lambda \Y_n) &  \numleqslant{i}  \mathbb{E}   \  \mathrm{tr} \   \exp \bigg( \lambda \Y_{n-1} + \log\exp \bigg(\frac{1}{2} \lambda^2 \bigg(\frac{  b-a  }{N-n} \bigg) ^2 \I  \bigg)  \bigg) \\
	 &  = \mathbb{E}   \  \mathrm{tr} \   \exp \bigg( \lambda \Y_{n-1} +\frac{\lambda^2}{2}  \bigg(\frac{  b-a  }{N-n} \bigg) ^2 \I  \bigg) \\
&\qquad \cdots \cdots	\\
&  \numleqslant{ii} \tre \bigg(\log \Ebb [e^{\lambda \Y_1}]  + \sum_{t=2}^{n} \frac{\lambda^2}{2} \bigg(\frac{  b-a  }{N-t}  \bigg)^2 \I \bigg).
 \numberthis
\end{align*}
where (i) follows from \cref{trace_exp_monotone,log_monotone}, and (ii) follows by applying the steps similar to obtain \cref{iteration_constant} for $n-2$ times.
 
To bound $\Ebb [e^{\lambda \Y_1}]$, we first note that 
\begin{align*}
  \Y_1  = \Z_{N-1}  &=  \frac{1}{N-1} \sum_{i=1}^{N-1}\bigg( \X_i - \mu\bigg) 
  \numequ{i}  \frac{1}{N-1} \bigg( N\mu - \X_N - (N-1) \mu \bigg)  
  = \frac{1}{N-1} \bigg( \mu - \X_N  \bigg) \label{HS_bound_Y1} ,
\end{align*}
where (i) follows because $N \mu = \sum_{i=1}^{N} \X_i$. 
Thus with  $a\I \preccurlyeq \X_i \preccurlyeq b\I$ and $a\I \preccurlyeq \mu \preccurlyeq b\I$, we obtain
\begin{align}
 - \frac{(b-a)}{N-1}\I  \preccurlyeq  \Y_1  \preccurlyeq \frac{(b-a)}{N-1}\I.
\end{align}
Applying the matrix Hoeffding lemma with \cref{HS_bound_Y1} and $\Ebb [Y_1] = \Ebb[Z_{N-1}] = 0$, we obtain
\begin{align} \label{HS_bound_Y1_final}
    \Ebb [e^{\lambda \Y_1}] \preccurlyeq  \exp \bigg(\frac{1}{2} \lambda^2 \bigg(\frac{b-1}{N-1}\bigg) ^2 \I  \bigg).
\end{align}

Substituting \cref{HS_bound_Y1_final} into \cref{HS_iteration}, we obtain
\begin{align*} \label{tre_lambda_y}
	 \mathbb{E} \  \mathrm{tr} \ \exp (\lambda \Y_n)  &\leqslant \tre \bigg(\sum_{t=1}^{n} \frac{\lambda^2}{2} \bigg(\frac{  b-a  }{N-t}  \bigg)^2 \I \bigg) \\
	 & = \tre \bigg( \frac{\lambda^2}{2} (b-a)^2 \sum_{t=1}^{n}  \bigg(\frac{ 1 }{N-t}  \bigg)^2 \I  \bigg) \\
	 & \numleqslant{i} \tre \bigg( \frac{\lambda^2}{2} (b-a)^2  \frac{n}{(N-n)^2} \bigg(1 - \frac{n-1}{N} \bigg) \I  \bigg), \numberthis  
\end{align*}
where (i) follows from \cref{HS_square_sum_inequality}.

Now let $m = N-n $, where  $1 \leqslant m \leqslant N-1$, and hence $\Y_n = \Z_{N-n}$. Thus, \cref{tre_lambda_y} implies
\begin{align*}
	 \mathbb{E} \  \mathrm{tr} \ \exp (\lambda \Z_{m})   
	 %&\leqslant \tre \bigg( \frac{\lambda^2}{2} (b-a)^2 \frac{(N-m)}{m^2}\bigg( \frac{m+1}{N} \bigg) I  \bigg)\\
	 & \leqslant   \tre \bigg( \frac{\lambda^2}{2} (b-a)^2 \frac{(m+1)}{m^2}\bigg( 1 - \frac{m}{N} \bigg) \I  \bigg).
\end{align*}
Substituting the above bound into \cref{HS_almost_done}, we obtain
\begin{align*} 			
P\bigg(\lambda_{\max}( \Z_n  ) \geqslant \epsilon \bigg) 
			& \leqslant \exp(- \lambda \epsilon) \ \tre \bigg( \frac{\lambda^2}{2} (b-a)^2 \frac{(n+1)}{n^2}\bigg( 1 - \frac{n}{N} \bigg) \I  \bigg) \\
%			 & \numleqslant{v} d \exp \bigg(\frac{\lambda^2}{2} (b-a)^2 \frac{(n+1)}{n^2} \bigg( 1 - \frac{n}{N} \bigg)\bigg) \exp(- \lambda \epsilon) \\
			 & = d \exp \bigg(\frac{\lambda^2}{2} (b-a)^2 \frac{(n+1)}{n^2} \bigg( 1 - \frac{n}{N} \bigg) - \lambda \epsilon \bigg),\ \numberthis
	\end{align*}
where the last step follows form the equation $\mathrm{ tr} (a\I) =  da$ for $\I \in \mathbb{R}^{d\times d}$. The proof is completed by minimizing the above bound with respect to  $\lambda > 0$, and then substituting the minimizer $ \lambda^\star = \frac{n\epsilon}{(b-a)^2(1+ \frac{1}{n}) ( 1 - \frac{n}{N}) }$.

\section{Proofs for SVRC under Sampling without Replacement}
\subsection{Proof of \Cref{SVRC_Without_Replacement}}
\begin{proof}
	 The idea of the proof is to apply the matrix concentration inequality for sampling without replacement that we developed in \Cref{Matrix_Hoeffding_Serfling_Inequality} to characterize the sample complexity in order to satisfy the inexactness condition $\norml{  \Hb_k - \nabla^2 F(\x_k) } \leqslant  \alpha \max \{\norml{\s_k}, \epsilon_1\} $ with the probability at least $1 - \zeta$. 
	 
	 We first note that
	  \begin{align*} 
	 \Hb_k - \nabla^2 F(\x_k) &\numequ{i} \tfrac{1}{|\xi_{H}(k)|}  \big[\! \textstyle\sum_{i \in \xi_{H}(k)} (\nabla^2 f_i(\x_k)  \!- \! \nabla^2 f_i(\tilde{\x})) \!\big]  \!+\!\nabla^2 F(\tilde{ \x}_k) - \nabla^2 F(\x_k) \\
	 &= \frac{1}{|\xi_{H}(k)|} \sum_{i \in\xi_{H}(k)}   \left(   \nabla^2 f_i(\x_k) -  \nabla^2 f_i(\tilde{\x})  +  \nabla^2 F(\tilde{\x}) - \nabla^2 F( \x_{k})   \right)
	 \end{align*}
	 where (i) follows from the definition of $\Hb_k$ in \Cref{SVRC}. In order to apply the concentration inequality (\Cref{Matrix_Hoeffding_Serfling_Inequality}) to bound $\Hb_k - \nabla^2 F(\x_k) $, we  define, for $1 \leqslant i \leqslant N$,  
	 \begin{align*}
	 \X_i = \nabla^2 f_i(\x_k) -  \nabla^2 f_i(\tilde{\x})  +  \nabla^2 F(\tilde{\x}) - \nabla^2 F( \x_{k}),
	 \end{align*}
	 which gives
	 \begin{align} \label{Whithou_replacement_hessian_samples_0}
	 	\Hb_k - \nabla^2 F(\x_k)  = \frac{1}{|\xi_{H}(k)|} \sum_{i \in\xi_{H}(k)} \X_i. 
	 \end{align}
	 
	 Moreover, we have $\mu \triangleq \frac{1}{N} \sum_{i =1}^{N}  \X_i = \mathbf{0}$, and
	 \begin{align*}
	 \sigma \triangleq \norml{\A_i} &= \norml{\nabla^2 f_i(\x_k) -  \nabla^2 f_i(\tilde{\x})  +  \nabla^2 F(\tilde{\x}) - \nabla^2 F( \x_{k})} \numleqslant{i} 2L_2 \norml{\x_{k} - \tilde{\x}},
	 \end{align*}
	 where (i) follows because $\nabla^2 f_i(\cdot)$ is $L_2$ Lipschitz, for $1 \leqslant i \leqslant N$.
    
     Thus, in order to satisfy  $\norml{  \Hb_k - \nabla^2 F(\x_k) } \leqslant  \alpha \max \{\norml{\s_k}, \epsilon_1\} $ with probability at least $1 - \zeta$, by \cref{Whithou_replacement_hessian_samples_0}, it is equivalent to satisfy $\norml{   \frac{1}{|\xi_{H}(k)|} \sum_{i \in\xi_{H}(k)} \X_i - \mu  } \leqslant \leqslant  \alpha \max \{\norml{\s_k}, \epsilon_1\}$ with probability at least $1 - \zeta$. We now apply \Cref{Matrix_Hoeffding_Serfling_Inequality} for $\X_i$,  and it is  sufficient to have:
	 \begin{align*}
	 2(d_1 + d_2) \exp \bigg( -  \frac{n \epsilon^2}{8 \sigma^2 (1+1/n) (1- n/N)}\bigg) \leqslant \zeta,
	 \end{align*}
	 which implies
	 \begin{align*}
	 \frac{n \epsilon^2}{8 \sigma^2 (1+1/n) (1- n/N)} \geqslant \log(\frac{2(d_1 + d_2)}{\zeta}).
	 \end{align*}
	 Using $(1+1/n) \leqslant 2$, it is sufficient to have:
	 \begin{align*}
	 \frac{n \epsilon^2}{16 \sigma^2   (1- n/N)} \geqslant \log(\frac{2(d_1 + d_2)}{\zeta}),
	 \end{align*}
	 which implies
	 \begin{align}
	 n \geqslant \frac{1}{\frac{1}{N} + \frac{\epsilon^2}{16 \sigma^2 \log(2(d_1+d_2)/\zeta)}}.
	 \end{align}
	 We then substitute $\sigma = 2L_2 \norml{\x_{k} - \tilde{\x}}$,   $\epsilon = \alpha \max \{\norml{\s_k}, \epsilon_1\}$, and $n = |\xi_{H}(k)|$, and obtain the required sample size to satisfy
	 \begin{align}
	 |\xi_{H}(k)| \geqslant \frac{1}{\frac{1}{N} + \frac{\alpha^2 \max \{\norml{\s_k}^2, \epsilon_1^2\}}{64 L_2^2 \norml{\x_{k} - \tilde{\x}}^2 \log(4d/\zeta)}}.
	 \end{align}
	 
	  We next bound the sample size $|\xi_{g}(k)|$ for the gradient, the proof follows the same procedure. We first define $\X_i \in \mathbb{R}^{d  \times 1}$ as
	 \begin{align} \label{Whithou_replacement_hessian_samples_1}
	 \X_i = \nabla f_i(\x_k) -  \nabla f_i(\tilde{\x})  +  \nabla F(\tilde{\x}) - \nabla F( \x_{k}),
	 \end{align}
	 and hence 
	 \begin{align}  \label{Whithou_replacement_hessian_samples_3}
	 	\g_k - \nabla  F(\x_k)  = \frac{1}{|\xi_{g}(k)|} \sum_{i \in\xi_{g}(k)} \X_i. 
	 \end{align}
	 Moreover, we have  $\mu = \frac{1}{N} \sum_{i \in \xi_{g}(k)} \A_i = \mathbf{0}$, and
	 \begin{align*}
	 \sigma \triangleq \norml{\A_i} &= \norml{\nabla  f_i(\x_k) -  \nabla  f_i(\tilde{\x})  +  \nabla  F(\tilde{\x}) - \nabla  F( \x_{k})} \numleqslant{i} 2L_1 \norml{\x_{k} - \tilde{\x}},
	 \end{align*}
	 where (i) follows because $\nabla  f_i(\cdot)$ is $L_1$ Lipschitz, for $1 \leqslant i \leqslant N$.  
	 
	 In order to satisfy $\norml{   \g_k - \nabla  F(\x_k)   } \leqslant \beta \max \{\norml{  \s_{k}}^2, \epsilon_1^2\}$ with probability at least $1 - \zeta$, by \cref{Whithou_replacement_hessian_samples_3}, it is equivalent to satisfy $\norml{    \frac{1}{|\xi_{g}(k)|} \sum_{i \in \xi_{g}(k)} \X_i - \mu   } \leqslant \beta \max \{\norml{  \s_{k}}^2, \epsilon_1^2\}$ with probability at least $1 - \zeta$. We then apply \Cref{Matrix_Hoeffding_Serfling_Inequality} for $\X_i$ in the way similar to that for bounding the sample size for Hessian, with $\sigma = 2L_1 \norml{\x_{k} - \tilde{\x}}$, $\mu = 0$, $\epsilon =\beta \max \{\norml{  \s_{k}}^2, \epsilon_1^2\}$, and $n = |\xi_{g}(k)|$, and obtain the required sample size to satisfy
	 \begin{align}
	 |\xi_{g}(k)| \geqslant \frac{1}{\frac{1}{N} + \frac{\beta^2 \max \{\norml{  \s_{k}}^4, \epsilon_1^4\}}{64 L_1^2 \norml{\x_{k} - \tilde{\x}}^2 \log(2(d+1)/\zeta)}}.
	 \end{align}
\end{proof}

\subsection{Proof of \Cref{SCR_Without_Replacement}}
\begin{proof}
	 The proof of \Cref{SCR_Without_Replacement} is similar to the proof of \Cref{SVRC_Without_Replacement}. We first define $\A_i \in \mathbb{R}^{d  \times d}$ as
	 \begin{align} \label{Whithou_replacement_SCR_hessian_samples_1}
	 \A_i = \nabla^2 f_i(\x_k) - \nabla^2 F( \x_{k}),
	 \end{align}
	 and hence $\mu = \frac{1}{N} \sum_{i \in \xi_{H}(k)} \A_i = \mathbf{0}$. Furthermore,
	 \begin{align*}
	 \sigma \triangleq \norml{\A_i} &= \norml{ \nabla^2 f_i(\x_k) - \nabla^2 F( \x_{k})} \numleqslant{i} 2L_1 ,
	 \end{align*}
	 where (i) follows from \Cref{assum: obj}.
	 
	 Let  $  \{\X_i\}_{i = 1}^{| \xi_{g}(k)|}= \{\A_i : i \in \xi_{H}(k) \}$, and we have
	 \begin{align} \label{Whithou_replacement_SCR_hessian_samples_3}
	 \frac{1}{|\xi_{H}(k)|} \sum_{i \in \xi_{H}(k)} \X_i - \mu \numequ{i} \frac{1}{|\xi_{H}(k)|} \sum_{i \in \xi_{H}(k)} \A_i \numequ{ii} \Hb_k - \nabla^2  F(\x_k),
	 \end{align}
	 where (i) follows from the fact that $\mu = 0$ and (ii) follows from the definition of $\Hb_k$ in \Cref{SVRC}. 
	 
    We then apply \Cref{Matrix_Hoeffding_Serfling_Inequality} for $\X_i$ with $\sigma = 2L_1 $, $\mu = 0$, $\epsilon =C_2 \norml{\x_{k+1} - \x_{k}}$, and $n = |\xi_{H}(k)|$, and obtain the require sampled size to satisfy
	 \begin{align}
	 |\xi_{H}(k)| \geqslant \frac{1}{\frac{1}{N} + \frac{C_2^2\norml{\x_{k+1} - \x_{k}}^2}{64 L_1^2   \log(4d/\zeta)}}.
	 \end{align}
	 
	 To bound the sample size of gradient, i.e., $|\xi_{g}(k)|$,  we follow the similar proof by constructing
	 \begin{align}
	 	 \A_i = \nabla f_i(\x_k) - \nabla F( \x_{k}),
	 \end{align}
	 
	  and applying  \Cref{Matrix_Hoeffding_Serfling_Inequality} with $\sigma = 2L_0$, $\mu = 0$, $\epsilon =C_1 \norml{\x_{k+1} - \x_{k}}^2$, and $n = |\xi_{g}(k)|$, and obtain the required sample size to satisfy
	  	 \begin{align}
	  |\xi_{g}(k)| \geqslant \frac{1}{\frac{1}{N} + \frac{C_1^2\norml{\x_{k+1} - \x_{k}}^4}{64 L_0^2   \log(2(d+1)/\zeta)}}.
	  \end{align}
	  
\end{proof}
%%%%%
\subsection{Proof of \Cref{th:no_rep_totalsample}} \label{Proof_theorem6}
%%%%%%% used with result
\begin{proof}
	Assume the algorithm terminates at iteration $k$, then the total Hessian complexity is given by 
	\begin{align*}
	m + \frac{kN}{m}  +  \sum_{i=0}^{k/m-1}  \sum_{j=1}^{m-1} |\xi_{H}(k) |  	&\numleqslant{i}   \frac{CkN}{m} +   \sum_{i=0}^{k/m-1} \sum_{j=1}^{m-1}  \frac{1}{\frac{1}{N} + \frac{\alpha^2 \max \{\norml{\s_k}^2, \epsilon_1^2\}}{64 L_2^2 \norml{\x_{i\cdot m + j} -  \x_{i\cdot m } }^2 \log(4d/\zeta)}} \\
	& \leqslant  	\frac{CkN}{m} +   \sum_{i=0}^{k/m-1} \sum_{j=1}^{m-1}  \frac{64 L_2^2 \norml{\x_{i\cdot m + j} -  \x_{i\cdot m  } }^2 \log(4d/\zeta)}{\alpha^2 \epsilon_1^2 }  \\
	&\numleqslant{ii} \frac{CkN}{m} +      \frac{64L_2^2}{\alpha^2 \epsilon_1^2}  \left(m^2k^{1/3}C^{2/3} \right)\log \left(\frac{ 4d}{\zeta}\right)   \\
	&\numleqslant{iii} C \log   \left(\frac{ 4d}{\zeta}\right)  \left( \frac{ N}{m\epsilon^{3/2}} +      \frac{  m^2  }{    \epsilon^{3/2} }    \right) = \frac{C}{\epsilon^{3/2}} \log \left(\frac{ 4d}{\zeta}\right)  \left(  \frac{ N}{m } +   m^2     \right)
	\end{align*}
	where (i) follows form  \Cref{SVRC_Without_Replacement}, and (ii) follows form \cref{bound_for_sum_1},
	(iii) follows from the fact that   $\zeta < 1$ and $d \geqslant 1$ which gives $\log \left(\frac{ 4d}{\zeta}\right) > 1$,  and  the fact that $\epsilon_1 = O(\epsilon^{1/2})$ such that $ k = O(\epsilon^{-3/2})$ according to \Cref{convergence_thm}.
	
	We minimize the above bound over $m$, substitute the minimizer $m^\star = N^{1/3}$, and follows the similar procedure in the proof of \cref{vector_norm_bound_7} to ensure  a successful event overall iteration with at least $1- \delta$, which gives that 
	\begin{align}
	\sum_{i=0}^{k} |\xi_{H}(k)|\leqslant   \frac{ C N^{2/3} }{\epsilon^{ 3/2}}\log \left(\frac{ 8d}{  \epsilon \delta}\right).
	\end{align}
	Thus, we have
	\begin{align}
	\sum_{i=0}^{k} |\xi_{H}(k)| = \tilde{O} \left(\frac{N^{3/2}}{\epsilon^{3/2}} \right).
	\end{align}
\end{proof}

\end{document}